\documentclass[11pt]{amsart}
\usepackage[utf8]{inputenc}
\usepackage[T1]{fontenc}
\usepackage{lmodern}
\usepackage{microtype}
\usepackage[leqno]{amsmath}
\usepackage{amssymb}
\usepackage{mathtools}
\usepackage{enumerate}
\usepackage{graphicx}
\usepackage[usenames,dvipsnames]{xcolor}
\usepackage[
colorlinks=true,linkcolor=NavyBlue,urlcolor=RoyalBlue,citecolor=PineGreen,bookmarks=true,bookmarksdepth=3,bookmarksopen=true,bookmarksopenlevel=2,unicode=true,linktocpage]{hyperref}

\numberwithin{equation}{section}
\numberwithin{figure}{section}

\newtheorem{theorem}{Theorem}[section]
\newtheorem{remark}[theorem]{Remark}
\newtheorem{lemma}[theorem]{Lemma}
\newtheorem{proposition}[theorem]{Proposition}

\newtheorem{definition}[theorem]{Definition}

\let\C\relax

\newcommand{\C}{\mathbf{C}}
\newcommand{\D}{\mathbf{D}}
\newcommand{\E}{\mathbf{E}}

\newcommand{\h}{\mathbf{H}}
\newcommand{\N}{\mathbf{N}}
\newcommand{\Z}{\mathbf{Z}}
\newcommand{\p}{\mathbf{P}}

\newcommand{\R}{\mathbf{R}}

\newcommand{\CE}{\mathcal {E}}
\newcommand{\CF}{\mathcal {F}}
\newcommand{\CI}{\mathcal {I}}
\newcommand{\CJ}{\mathcal {J}}

\newcommand{\CL}{\mathcal {L}}

\newcommand{\CS}{\mathcal {S}}

\newcommand{\SLE}{{\rm SLE}}
\newcommand{\CLE}{{\rm CLE}}

\newcommand{\dist}{\mathrm{dist}}

\newcommand{\one}{{\bf 1}}

\newcommand{\wt}{\widetilde}

\newcommand{\ol}{\overline}
\newcommand{\ul}{\underline}

\newcommand{\quant}[3][]{{{\mathfrak q}}_{#3}\if\relax\detokenize{#1}\relax\else^{#1}\fi(#2)}
\newcommand{\median}[2][]{{\mathfrak m}_{#2}\if\relax\detokenize{#1}\relax\else^{#1}\fi}
\newcommand{\mediant}[2][]{\wt{\mathfrak m}_{#2}\if\relax\detokenize{#1}\relax\else^{#1}\fi}

\newcommand{\Fd}{\mathfrak d}
\newcommand{\met}[3]{\Fd(#1,#2;#3)}

\newcommand{\metres}[4]{\Fd^{#1}(#2,#3;#4)}

\newcommand{\geoexp}{{\alpha_{\mathrm{g}}}}
\newcommand{\resexp}{{\alpha_{\mathrm{r}}}}

\newcommand{\FR}{{\mathfrak R}}

\newcommand{\rmet}[3]{\FR(#1,#2;#3)}

\newcommand{\rmetres}[4]{\FR^{#1}(#2,#3;#4)}

\newcommand{\meas}[2]{\mu(#1; #2)}

\newcommand*{\graph}{\mathfrak{G}}
\newcommand*{\vertexset}{\mathfrak{V}}
\newcommand*{\edgeset}{\mathfrak{E}}

\newcommand*{\approxgraph}[2][]{\graph_{#2}\if\relax\detokenize{#1}\relax\else^{#1}\fi}
\newcommand*{\approxvtcs}[2][]{\vertexset_{#2}\if\relax\detokenize{#1}\relax\else^{#1}\fi}
\newcommand*{\approxedges}[2][]{\edgeset_{#2}\if\relax\detokenize{#1}\relax\else^{#1}\fi}

\DeclarePairedDelimiter\abs{\lvert}{\rvert}
\DeclarePairedDelimiter\norm{\lVert}{\rVert}

\newcommand*{\eqdef}{=\mathrel{\mathop:}}

\newcommand*{\mmiddle}[1]{\mathrel{}\middle#1\mathrel{}}

\newcommand*{\sle}[1]{$\SLE_{#1}$}
\newcommand*{\slek}{\sle{\kappa}}

\newcommand*{\slekr}[1]{$\SLE_{\kappa}(#1)$}

\newcommand*{\slekrb}[2]{$\SLE_{\kappa}^{#2}(#1)$}
\newcommand*{\slefmu}[1]{$\SLE_{4}^{0,#1}(-2)$}
\newcommand*{\cle}[1]{$\CLE_{#1}$}
\newcommand*{\clek}{\cle{\kappa}}

\newcommand{\markeddomain}[1]{{\mathfrak D}_{#1}}
\newcommand{\eldomain}[1]{{\mathfrak D}_{#1}^{\mathrm{ext}}}
\newcommand{\ildomain}[1]{{\mathfrak D}_{#1}^{\mathrm{int}}}

\newcommand{\mcclelaw}[1]{\p_{(#1)}^{\CLE_{\kappa}}}
\newcommand{\mccleexp}[1]{\E_{(#1)}^{\CLE_{\kappa}}}

\newcommand{\domainpair}[1]{{\mathfrak {P}}_{#1}}

\newcommand{\outside}{{\mathrm{out}}}
\newcommand{\inside}{{\mathrm{in}}}

\newcommand{\separated}{{\mathrm{sep}}}
\newcommand{\resampled}{{\mathrm{res}}}

\newcommand*{\metregions}[1][]{\mathfrak{C}\if\relax\detokenize{#1}\relax\else_{#1}\fi}

\newcommand*{\diamE}{\operatorname{diam_E}}
\newcommand*{\BE}{B_{\mathrm{E}}}
\newcommand*{\dpathY}[1][]{d_{\mathrm{path}}\if\relax\detokenize{#1}\relax\else^{#1}\fi}
\newcommand{\Bpath}{B_{\mathrm{path}}}

\newcommand*{\len}[2]{L_{#1}(#2)}

\newcommand*{\paths}[4]{P(#1,#2;#3;#4)}

\newcommand{\lmet}[1]{L_{\Fd}(#1)}

\newcommand{\Bgeo}{B_{\mathrm{geo}}}
\newcommand{\Bres}{B_{\mathrm{res}}}

\newcommand{\dcle}{d_\CLE}

\newcommand{\measapprox}[3]{\mu_{#1}(#2; #3)}
\newcommand{\leb}{{\mathrm {Leb}}}
\newcommand{\cen}{{\mathrm {cen}}}

\newcommand{\bad}{{\mathrm{bad}}}

\newcommand*{\Esep}{E^{\mathrm{sep}}}
\newcommand*{\dsep}{s_{\mathrm{sep}}}

\title[CLE Minkowski content]{Minkowski content construction of\\ the CLE gasket measure}

\usepackage[foot]{amsaddr}
\author{Jason Miller}
\author{Yizheng Yuan}
\address{Department of Pure Mathematics and Mathematical Statistics, University of Cambridge}
\address{Faculty of Mathematics, University of Vienna}

\date{\today}

\begin{document}

\begin{abstract}
We show that the canonical conformally covariant measure on the conformal loop ensemble (CLE$_\kappa$) gasket/carpet, previously constructed indirectly by the first co-author and Schoug, can be realized as the limit of several natural approximation schemes.  These include the Euclidean Minkowski content and its box-count variants, the properly renormalized number of dyadic squares that intersect the gasket, and the properly renormalized minimal number of balls of radius $\delta$ necessary to cover the gasket with respect to both its canonical geodesic and resistance metrics.  This in particular allows us to identify the CLE$_6$ gasket measure with the conformally covariant measure constructed by Garban--Pete--Schramm as a scaling limit of the number of vertices in a macroscopic critical percolation cluster on the triangular lattice.  Along the way, we show that the CLE gasket measure of every fixed compact set has finite moments of all orders; previously this was only known for first moments.
\end{abstract}

\maketitle

\parindent 0 pt
\setlength{\parskip}{0.20cm plus1mm minus1mm}

\section{Introduction}
\label{sec:introduction}

\subsection{Overview}
\label{subsec:overview}

In this work, we will consider the \emph{conformal loop ensembles} ($\CLE_\kappa$) \cite{s2009cle, sw2012cle}, which are the loop version of $\SLE_\kappa$ \cite{s2000sle}. They arise as the joint scaling limit of all of the interfaces of various critical discrete models from two-dimensional statistical mechanics.  For example, the cases $\kappa=3,16/3,6,8$ were respectively shown to be the scaling limit of the critical Ising model \cite{bh2019ising}, FK Ising model \cite{ks2019fkising}, percolation model \cite{s2001percolation,cn2006percolationcle6}, and the uniform spanning tree \cite{lsw2004lerw}. For each value of the parameter $\kappa \in [8/3,8]$, a $\CLE_\kappa$ is a random collection of loops that live in a simply connected domain in $\C$.  When $\kappa \in (8/3,8)$, a $\CLE_\kappa$ consists of countably infinitely many loops, each of which looks locally like an $\SLE_\kappa$. In the case $\kappa \in (8/3,4]$ the loops are simple, do not intersect each other, or the domain boundary, while for $\kappa \in (4,8)$ the loops are self-intersecting (but not space-filling), intersect each other, and the domain boundary. The borderline cases $\kappa=8/3$ and $\kappa=8$ are respectively an empty set and a single space-filling loop.

In the non-nested $\CLE_\kappa$, there are no further loops inside each of its loops, whereas the nested $\CLE_\kappa$ is obtained by iteratively sampling a conditionally independent non-nested $\CLE_\kappa$ in each of the components inside the loops of a non-nested $\CLE_\kappa$. Conversely, a non-nested $\CLE_\kappa$ is obtained by considering the collection of outermost loops of a nested $\CLE_\kappa$.

Suppose that $\Gamma$ is a non-nested $\CLE_\kappa$. The carpet (resp.\ gasket) $\Upsilon$ of $\Gamma$ if $\kappa \in (8/3,4]$ (resp.\ $\kappa \in (4,8)$) is the closure of the set of points not inside any of its loops.  (The reason for this distinction in terminology is that for $\kappa \in (8/3,4]$ it is a random analog of the Sierpinski carpet while for $\kappa \in (4,8)$ it is a random analog of the Sierpinski gasket.)  The (Euclidean) dimension of the $\CLE_\kappa$ carpet/gasket is given by \cite{ssw2009radii,msw2014dimension}
\begin{equation}
\label{eqn:cle_dim}
\dcle = 2 - \frac{(8-\kappa)(3\kappa-8)}{32\kappa} = 1 + \frac{2}{\kappa} + \frac{3\kappa}{32} .
\end{equation}
Similarly, if $\Gamma$ is a nested $\CLE_\kappa$ and $\CL \in \Gamma$ is one of its loops, then one can define the carpet/gasket~$\Upsilon_\CL$ inside $\CL$ as the carpet/gasket of the non-nested $\CLE_\kappa$ in the next level inside $\CL$.

In \cite{ms2022clemeasure}, it was shown that there exists a unique (up to a deterministic constant) measure $\meas{\cdot}{\Gamma}$ on $\Upsilon \cap D$ satisfying the following properties.
\begin{itemize}
 \item $\E[\meas{K}{\Gamma}] < \infty$ for each compact set $K \subseteq D$.
 \item Let $U \subseteq D$ be open, simply connected, and let $U^* \subseteq U$ be the set of points that are not on or inside a loop $\CL \in \Gamma$ with $\CL \cap D\setminus U \neq \emptyset$. For each connected component $V$ of $U^*$, let $\Gamma_V \subseteq \Gamma$ be the loops contained in $\ol{V}$. Then the measure $\meas{\cdot \cap V}{\Gamma}$ is determined by $\Gamma_V$ and almost surely
 \[ \meas{\varphi(A)}{\Gamma_V} = \int_A \abs{\varphi'(z)}^{\dcle} \meas{dz}{\varphi^{-1}(\Gamma_V)} \]
 where $\varphi\colon D \to V$ is a conformal transformation.
\end{itemize}

The construction of the measure given in \cite{ms2022clemeasure} is indirect and is based on the construction of the corresponding measure in the setting of Liouville quantum gravity \cite{ms2021nonsimplelqg,ms2022simplelqg}, which in turn relies on the theory of growth-fragmentation processes.  The purpose of this work is to show that the measure of \cite{ms2022clemeasure} admits several direct constructions, including by Minkowski content.  This parallels the development of the so-called natural parameterization of $\SLE_\kappa$, which we recall was first constructed in \cite{ls2011natural} for a small range of $\kappa$ values using an indirect method.  The indirect construction was then extended to all $\kappa$ values in \cite{lz2013natural}.  It was later shown that the natural parameterization corresponds to Minkowski content in \cite{lr2015natural}.

In the particular case that $\kappa=6$, a conformally covariant measure defined on $\CLE_6$ was constructed in \cite{gps-measure} as the scaling limit of the counting measure on the critical percolation converging to~$\CLE_6$.  It is not immediate, however, from the construction given in \cite{ms2022clemeasure} that the measure from \cite{ms2022clemeasure} agrees with (a variant of) the measure from \cite{gps-measure}.  This will be one consequence of the present work.  This, in turn, will play an important role in showing that the scaling limit of random walk on two-dimensional percolation clusters converges in the scaling limit \cite{dmmy2026percolation} to the canonical Brownian motion on the $\CLE_6$ gasket \cite{amy2025tightness,my2025resuniqueness}.

A similar construction of a measure on the \cle{16/3} gaskets in the full plane was carried out in \cite{cck-measure-fkising} starting from the scaling limit of the FK-Ising model, but only scale-covariance of their measure was proved. Our main result also implies that their measure agrees with the canonical conformally covariant gasket measure on \cle{16/3}.

Finally, we remark that an analogous construction of the \cle{4} carpet measure (as in our Theorem~\ref{th:minkowski}\eqref{it:box_count}) was also recently proved in \cite{alhl-ising-cle4} as a scaling limit of another discrete model. In the context of Liouville quantum gravity, a construction of the volume measure from the intrinsic metric (as in our Theorem~\ref{th:minkowski}\eqref{it:geo_count}) was proved in \cite{gs-lqg-minkowski}.

\subsection{Main results}
\label{subsec:main_results}

We now state our main results.  Let $D \subseteq \C$ be a simply connected domain. For each $k \in \N$ we let $\CS_k$ be the set of squares with corners in $(2^{-k} \Z)^2$ and side length $2^{-k}$ \emph{that have positive distance to $\C \setminus D$}. We let $\CS = \bigcup_{k \in \N} \CS_k$. For each square $Q$ and $\lambda > 0$, we denote by $\lambda Q$ the square with the same center as $Q$ and $\lambda$ times its side length.

Fix $\kappa \in (8/3,8)$, let $\Gamma$ be a nested \clek{} in $D$, and let $\Upsilon_{\partial D}$ be its outermost gasket\footnote{We will use the term ``gasket'' throughout the article even though in the case $\kappa \in (8/3,4]$ it is commonly referred to as ``carpet''.} which has $\partial D$ as its exterior boundary. For each loop $\CL \in \Gamma$ we let $\Upsilon_\CL$ be the interior gasket whose exterior boundary is~$\CL$. That is, $\Upsilon_\CL$ is set of points on or inside\footnote{Here, we say that a point $z$ is inside a loop $\CL$ if the winding number of $\CL$ about $z$ is $\pm 1$ (depending on the choice of its orientation).} $\CL$ and not inside another loop of $\Gamma$ inside $\CL$.  We note that we can define a measure $\meas{\cdot}{\Upsilon_\CL}$ on each interior gasket $\Upsilon_\CL$ by summing the measure from \cite{ms2022clemeasure} over the components inside $\CL$. (This notation is justified: Theorem~\ref{th:minkowski} below implies that the measure depends only on the gasket as a set in the Euclidean plane.)

Let $\Upsilon$ be a gasket of $\Gamma$. For each $B \in \CS$ and either $k \in \N$ or $\delta \in (0,1]$ we consider the following approximations of the \clek{} gasket measure.
\begin{enumerate}[(i)]
 \item\label{it:box_count} The box count
 \[ \measapprox{k}{B}{\Upsilon} = 2^{-\dcle k} \sum_{Q \in \CS_k ,\, Q \subseteq B} \one_{Q \cap \Upsilon \neq \varnothing} . \]
 \item\label{it:gps_box_count} The variant of the box count from \cite{gps-measure}
 \[ \measapprox{k}{B}{\Upsilon} = 2^{-\dcle k} \sum_{Q \in \CS_k ,\, Q \subseteq B} \one_{2Q \cap \Upsilon \neq \varnothing} . \]
 \item\label{it:minkowski} The (Euclidean) Minkowski content
 \[ \measapprox{\delta}{B}{\Upsilon} = \delta^{\dcle-2} \,\leb{\left( \bigcup_{x\in B \cap \Upsilon} \BE(x,\delta) \right)} \]
 where $\BE(x,\delta) = \{y \in \C : \norm{y-x}_2 < \delta\}$ denotes the Euclidean ball.
\end{enumerate}
The approximation schemes considered in~\eqref{it:box_count}--\eqref{it:minkowski} are based on \emph{Euclidean} approximations.  The following two approximation schemes are based on quantities which are \emph{intrinsic} to the $\CLE_\kappa$ gasket, in particular its canonical \emph{geodesic} \cite{my2025geouniqueness} and \emph{resistance} \cite{my2025resuniqueness} metrics which were constructed in the case $\kappa \in (4,8)$.  (We recall the basics of these metrics in Section~\ref{sec:preliminaries}. We remark that the construction of the resistance metric in \cite{my2025resuniqueness} uses only the Propositions~\ref{pr:measure_ub} and~\ref{pr:measure_lb} of the present paper.) In the following, we let $\Bgeo(x,r)$ (resp.\ $\Bres(x,r)$) denote a ball with respect to the geodesic (resp.\ resistance) metric, and the gasket $\Upsilon$ is regarded as a metric space embedded in the plane.
\begin{enumerate}[(i)]
\setcounter{enumi}{3}
 \item\label{it:geo_count} The geodesic metric ball count
 \[ \measapprox{\delta}{B}{\Upsilon} = \delta^{\dcle} \min\left\{ n \in \N : \text{There exist $x_1,\ldots,x_n \in \Upsilon$ such that $B \cap \Upsilon \subseteq \bigcup_i \Bgeo(x_i,\delta^\geoexp)$} \right\} \]
 where $\geoexp$ is the conformal covariance exponent for the geodesic metric.
 \medskip
 \item\label{it:res_count} The resistance metric ball count
 \[ \measapprox{\delta}{B}{\Upsilon} = \delta^{\dcle} \min\left\{ n \in \N : \text{There exist $x_1,\ldots,x_n \in \Upsilon$ such that $B \cap \Upsilon \subseteq \bigcup_i \Bres(x_i,\delta^\resexp)$} \right\} \]
 where $\resexp$ is the conformal covariance exponent for the resistance metric.
\end{enumerate}
For $\kappa \in (8/3,4)$, the canonical geodesic metric on the carpets of a whole-plane \clek{} were recently constructed in \cite{mt-geodesic}. The approximation~\eqref{it:geo_count} is therefore also defined in this regime (for small enough $\delta$) due to the locality of the geodesic metric and the local absolute continuity between \clek{} carpets in different domains (see Lemma~\ref{le:cle_abs_cont}). Wherever we consider~\eqref{it:geo_count}--\eqref{it:res_count} in the article, we will implicitly assume that $\kappa$ is in the regime where the respective metric has been defined.

In the remainder of the article, we will write $\measapprox{k}{B}{\Upsilon} = \measapprox{2^{-k}}{B}{\Upsilon}$ also in the cases~\eqref{it:minkowski}--\eqref{it:res_count} where it is defined for all real $k > 0$. In the statements, when writing $k>0$, we implicitly restrict to $k \in \N$ in the cases~\eqref{it:box_count}--\eqref{it:gps_box_count}. (One could define these variants also for real $k>0$ although it is a bit unnatural.)

Note that all the approximate measures~\eqref{it:box_count}--\eqref{it:res_count} satisfy the scaling
\begin{equation}\label{eq:approxmeas_scaling}
 \measapprox{j+k}{2^{-j}B}{2^{-j}\Upsilon} = 2^{-j\dcle}\measapprox{k}{B}{\Upsilon} .
\end{equation}

\begin{remark}
 In fact, the only key properties of the approximate measures that we need are the scaling~\eqref{eq:approxmeas_scaling}, the approximate additivity (Lemma~\ref{le:approx_additivity}), and finite moments (Lemma~\ref{le:approxmeas_moments}). One can phrase our main result Theorem~\ref{th:minkowski} under an abstract set of conditions for $\measapprox{k}{\cdot}{\Upsilon}$, but for the sake of simplicity and readability we choose to be explicit.
\end{remark}

\begin{theorem}\label{th:minkowski}
 Let $\Gamma$ be a nested \clek{} in $D$. For each of its gaskets $\Upsilon$, let $\meas{\cdot}{\Upsilon}$ be the gasket measure from \cite{ms2022clemeasure}. Let $\measapprox{k}{\cdot}{\Upsilon}$ be one of the approximate measures~\eqref{it:box_count}--\eqref{it:res_count} defined above. There exists a deterministic constant $c>0$ (depending only on $\kappa$ and on the approximation scheme) such that almost surely for each of the gaskets of $\Gamma$ we have
 \begin{equation}\label{eq:measure_convergence}
  \lim_{k \to \infty} \measapprox{k}{B}{\Upsilon} = c\meas{B}{\Upsilon}
  \quad\text{for each } B \in \CS .
 \end{equation}
\end{theorem}

Note that the measure $\meas{\cdot}{\Upsilon}$ is determined by its values on $\CS$, so Theorem~\ref{th:minkowski} states how the measure $\meas{\cdot}{\Upsilon}$ is explicitly obtained from the approximations $(\measapprox{k}{\cdot}{\Upsilon})$. Moreover, it follows that~\eqref{eq:measure_convergence} holds for each Borel set $B \Subset D$ with $\meas{\partial B}{\Upsilon} = 0$. In particular, by Proposition~\ref{pr:measure_ub} below, this holds for each $B$ such that $\partial B \cap \Upsilon$ has Hausdorff dimension smaller than~$\dcle$. For the approximations~\eqref{it:box_count}--\eqref{it:gps_box_count}, each $\measapprox{k}{\cdot}{\Upsilon}$ defines exactly a measure, so in that case Theorem~\ref{th:minkowski} implies that $\measapprox{k}{\cdot}{\Upsilon}$ converges vaguely to $c\meas{\cdot}{\Upsilon}$ as $k \to \infty$.

\begin{remark}
A similar characterization has been proved for a variant of the \cle{6} gasket measure in \cite{gps-measure} where the renormalization factor is implicitly given in terms of the \cle{6} one-arm probability $\alpha_1(2^{-k},1)$. Our result, in contrast, gives the exact renormalization factor which is just a power function. We will see that we can obtain this because we know that the gasket measure satisfies the exact scaling property. As a consequence, comparing this to the result in \cite{gps-measure}, we will obtain that $\alpha_1(2^{-k},1) = c2^{-(5/48)k}+o(2^{-(5/48)k})$ for some constant $c>0$. See \cite{dmmy2026percolation} for details.
\end{remark}

Along the way, we will prove the following regularity properties of the \clek{} measure which are of independent interest.

\begin{proposition}\label{pr:measure_ub}
Let $\Gamma$ be a nested \clek{} in a simply connected domain $D \subseteq \C$. For each of the gaskets $\Upsilon$ of $\Gamma$, let $\meas{\cdot}{\Upsilon}$ denote the gasket measure on $\Upsilon$. Fix any $\delta > 0$, $a>0$, and a compact set $K \subseteq D$. For each $r>0$, let $E_r$ be the event that
\[ r^{\dcle+a} \le \meas{\BE(z,r)}{\Upsilon} \le r^{\dcle-a} \quad\text{for each } z \in \Upsilon \cap K \]
and for each gasket $\Upsilon$ of $\Gamma$ with $\diamE(\Upsilon) \ge \delta$. Then
\[ \p[(E_r)^c] = o^\infty(r) \quad\text{as } r \searrow 0 . \]
\end{proposition}

We will also prove a stronger version of the lower bound which will be useful in \cite{my2025resuniqueness}. In order to give the statement, we first recall a definition. Let $\Gamma$ be a \clek{} in $D$, and let $\Upsilon$ be one of its gaskets. We say that a path $\gamma$ is \emph{admissible} for $\Gamma$ if it is contained in $D$ and does not cross any loop of $\Gamma$. We define the metric $\dpathY$ on $\Upsilon$ by\footnote{Strictly speaking, when $\kappa \in (4,8)$, the metric $\dpathY$ is defined on the set of \emph{prime ends} of $\Upsilon$. See Section~\ref{subsec:cle_metrics} for more details.}
\begin{equation}\label{eq:dpath}
 \dpathY(x,y) = \inf\{ \diamE(\gamma) : \gamma \text{ admissible path between $x$ and $y$} \} .
\end{equation}
We let $\Bpath(x,r)$ denote a ball in the metric $\dpathY$, and let $\partial\Bpath(x,r)$ be its boundary with respect to the metric~$\dpathY$.

\begin{proposition}\label{pr:measure_lb}
Let $\Gamma$ be a nested \clek{} in a simply connected domain $D \subseteq \C$. For each of the gaskets $\Upsilon$ of $\Gamma$, let $\meas{\cdot}{\Upsilon}$ denote the gasket measure on $\Upsilon$. Fix any $a>0$ and a compact set $K \subseteq D$. For each $r>0$, let $E_r$ be the event that the following holds. Let $\Upsilon$ be any of the gaskets of~$\Gamma$, let $x \in \Upsilon \cap K$ and $w \in \partial\Bpath(x,r)$. Let $B_{x,w} \subseteq \ol{\BE}(x,r)$ be a simply connected Borel set containing all simple admissible paths within $\ol{\BE}(x,r)$ from $x$ to $w$. Then
\[ \meas{B_{x,w}}{\Upsilon} \ge r^{\dcle+a} \quad\text{for each such $x,w$, and $\Upsilon$.} \]
Then
\[ \p[(E_r)^c] = o^\infty(r) \quad\text{as } r \searrow 0 . \]
\end{proposition}

\begin{figure}[ht]
\centering
\includegraphics[width=0.5\textwidth]{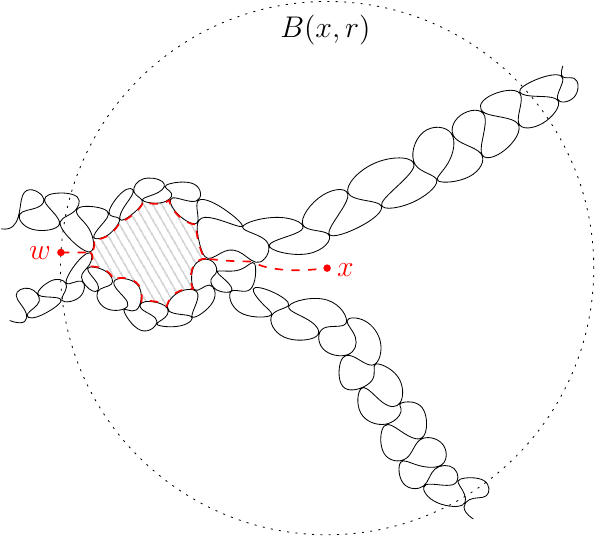}
\caption{Illustration of the statement of Proposition~\ref{pr:measure_lb} in the case $\kappa \in (4,8)$. Shown in red are two simple admissible paths between $x$ and $w \in \partial\Bpath(x,r)$ in a gasket $\Upsilon$ of $\Gamma$. The region $B_{x,w}$ is a simply connected set in the Euclidean plane and contains the two red paths, hence needs to contain the shaded region between them. Note that when $\kappa \in (4,8)$, there are points in $\ol{\Bpath}(x,r)$ that are separated from $x$ and $w$ by a single point; these points are not necessarily contained in $B_{x,w}$. (In the case $\kappa \in (8/3,4]$, we always have $B_{x,w} \supseteq \ol{\Bpath}(x,r)$.)}
\label{fi:region_containing_bubble}
\end{figure}

\subsection*{Outline}

The remainder of this article is structured as follows.  In Section~\ref{sec:preliminaries}, we will collect some preliminaries.  Next, in Section~\ref{sec:moment_finiteness} we will prove the finiteness of the moments of the $\CLE_\kappa$ gasket measure as well as Propositions~\ref{pr:measure_ub} and~\ref{pr:measure_lb}.  Finally, in Section~\ref{sec:minkowski_content} we will complete the proof of Theorem~\ref{th:minkowski}.

Let us now give some further ideas regarding the heart of the argument of the proof of Theorem~\ref{th:minkowski}.  Theorem~\ref{th:minkowski} can be seen as a strong law of large numbers for the \clek{} gasket measure. Indeed, suppose that $B \in \CS_l$ and $j > l$ so that we can divide $B$ into a collection of subsquares in $\CS_j$.  Then we have
\[
 \measapprox{k}{B}{\Upsilon} \approx \sum_{Q \in \CS_j ,\, \Upsilon \cap Q \neq \varnothing} \measapprox{k}{Q}{\Upsilon}.
\]
By the scaling property~\eqref{eq:approxmeas_scaling}, we have
\[
 \measapprox{k}{Q}{\Upsilon} \approx 2^{-j\dcle}\measapprox{k-j}{2^j Q}{2^j \Upsilon} .
\]
If we pretend the collection $(\measapprox{k}{Q}{\Upsilon})$ to be i.i.d., then we expect that
\[
 \measapprox{k}{B}{\Upsilon} \approx \#\{Q \in \CS_j ,\, \Upsilon \cap Q \neq \varnothing\} 2^{-j\dcle}\E[\measapprox{k-j}{B}{\Upsilon}] .
\]
Similarly, we expect that
\[
 \meas{B}{\Upsilon} \approx \#\{Q \in \CS_j ,\, \Upsilon \cap Q \neq \varnothing\} 2^{-j\dcle}\E[\meas{B}{\Upsilon}] .
\]
Therefore the constant $c>0$ in Theorem~\ref{th:minkowski} represents the ratio $\E[\measapprox{k}{B}{\Upsilon}]/\E[\meas{B}{\Upsilon}]$ in an appropriate setup.

The main challenge in proving it is the lack of exact independence. Our main tool is \cite{amy-cle-resampling} which establishes sufficient spatial independence. This will enough to prove in Section~\ref{se:meas_comparability} that $\measapprox{k}{B}{\Upsilon}$ and $\meas{B}{\Upsilon}$ are comparable up to some deterministic constants $c_* \le c^*$. In Section~\ref{se:meas_limit}, we then use the self-similarity to argue that (at least along subsequences) the constants $c_*,c^*$ can be improved until they agree. Up to this step, this strategy is reminiscent of the proof of the main result of \cite{my2025geouniqueness}, however it only gives us the convergence in probability, and up to some additional scaling factors~$(c_k)$ for~$\measapprox{k}{\cdot}{\Upsilon}$. We then need to improve this to almost sure convergence and obtain the exact scaling factors in the approximation $\measapprox{k}{\cdot}{\Upsilon}$. This is carried out in Section~\ref{se:convergence_improve}.

\subsection*{Acknowledgements} J.M. and Y.Y.\ were supported by ERC starting grant SPRS (804116) and from ERC consolidator grant ARPF (Horizon Europe UKRI G120614). Y.Y.\ in addition received support from the Royal Society.

\subsection*{Notation} Let $\delta \in (0,1]$. We let $O(\delta)$ denote a function that is bounded by $c\delta$ for some constant $c>0$. We let $o^\infty(\delta)$ denote a function such that for each $\beta>0$ there is a constant $c>0$ such that the function is bounded by $c\delta^\beta$. We write $a \lesssim b$ to denote that $a \le cb$ for some constant $c>0$, and we write $a \asymp b$ to denote that $a \lesssim b$ and $b \lesssim a$.

\section{Preliminaries}
\label{sec:preliminaries}

The purpose of this section is to collect some preliminaries.  First, we will review some basic facts about $\SLE$ in Section~\ref{subsec:sle}.  Next, we will review the construction and basic facts about $\CLE$ in Section~\ref{subsec:cle} and its multichordal variant. In Section~\ref{subsec:cle_metrics} we will review the basics of the $\CLE$ geodesic and resistance metrics (this is only relevant for parts~\eqref{it:geo_count} and~\eqref{it:res_count} of Theorem~\ref{th:minkowski}).  Finally, in Section~\ref{subsec:independence_across_scales} we review results about independence across scales for $\CLE$.

\subsection{Schramm-Loewner evolution}
\label{subsec:sle}

The starting point for the definition of $\SLE$ is the chordal Loewner equation
\begin{equation}
\label{eqn:loewner_ode}
\partial_t g_t(z) = \frac{2}{g_t(z) - W_t},\quad g_0(z) = z.
\end{equation}
Here, $W \colon \R_+ \to \R$ is a continuous function, and for each $z \in \h$ the solution $(g_t(z))$ is defined up to the time $\tau_z = \inf\{t \geq 0 : g_t(z) = W_t\}$.  If we let $\h_t = \{ z  \in \h : t < \tau_z\}$, then $g_t$ is the unique conformal map $\h_t \to \h$ with $g_t(z)  - z \to 0$ as $z \to \infty$.  If we let $K_t = \h \setminus \h_t$ then $K_t$ is a compact $\h$-hull meaning that $\ol{K_t}$ is compact and $\h \setminus K_t$ is simply connected.

The particular case where $W = \sqrt{\kappa} B$ and $B$ is a standard Brownian motion defines $\SLE_\kappa$ \cite{s2000sle}. More precisely, the $\SLE_\kappa$ process in $\h$ from $0$ to $\infty$ is a continuous curve $\eta$ in $\ol{\h}$ that \emph{generates} the family $(K_t)$ of compact $\h$-hulls associated with $W = \sqrt{\kappa} B$, meaning that $\h \setminus K_t$ is the unbounded component of $\h \setminus \eta([0,t])$. That there is almost surely such a continuous curve was proved in \cite{rs2005basic} for $\kappa \neq 8$, and in \cite{lsw2004lerw} for $\kappa = 8$ as a consequence of the convergence of the uniform spanning tree to $\SLE_8$ (see also \cite{am2022sle8} for a proof of the continuity of $\SLE_8$ based only on continuum methods).

The $\SLE_\kappa(\rho)$ processes, first introduced in \cite[Section~8.3]{lsw2003restriction}, are an important variant of $\SLE_\kappa$ in which one keeps track of additional marked points.  In this paper, we will need the case in which there is a single extra marked point $x_0$ on the boundary.  In this case, the process defined by solving~\eqref{eqn:loewner_ode} with $W$ taken to be the solution to the SDE
\begin{equation}\label{eq:slekr}
dW_t = \sqrt{\kappa} dB_t + \frac{\rho}{W_t - V_t} dt,\quad dV_t = \frac{2}{V_t - W_t} dt, \quad V_0 = x_0.
\end{equation}
The point $x_0$ is the so-called \emph{force point} and $\rho$ is the so-called \emph{weight}.  When $\rho$ is positive (resp.\ negative), the process is repelled from (resp.\ attracted to) its force point.  The continuity of the $\SLE_\kappa(\rho)$ processes for $\rho > -2$ was proved in \cite{ms2016ig1}.

For $\rho > -2-\kappa/2$, one can also define \emph{generalised $\SLE_\kappa(\rho)$ processes}, denoted by \slekrb{\rho}{\beta} where $\beta \in [-1,1]$ is an asymmetry parameter. See \cite{s2009cle} for details (see also \cite[Section~2.1.1]{amy-cle-resampling} for a summary). The continuity of these processes are proved for $\rho \in [\kappa/2-4,-2)$ in \cite{ms2019lightcone}, and for $\rho \in (-2-\kappa/2,\kappa/2-4)$ in \cite{msw2017cleperc}.

$\SLE_\kappa$ and $\SLE_\kappa(\rho)$ processes in general simply connected domains $D \subseteq \C$ are defined by conformally mapping the processes $\h \to D$ where $0$ (resp.\ $\infty$) is mapped to the starting (resp.\ ending) point of the curve.  The value $\rho = \kappa-6$ is special due to the \emph{target invariance} of the $\SLE_\kappa(\kappa-6)$ processes \cite{sw2005coordinate}.  This means that for all $x,y \in \partial \h \cup \{ \infty\}$ the law of an $\SLE_\kappa(\kappa-6)$ process in $\h$ from $0$ to $x$ with force point at $0^+$ stopped at the first time it disconnects $x$ from $y$ is the same as that of an $\SLE_\kappa(\kappa-6)$ process in $\h$ from $0$ to $y$ with force point at $0^+$ stopped at the first time it disconnects $y$ from $x$.  This property is crucial for the construction of $\CLE$.

\subsection{Conformal loop ensembles}
\label{subsec:cle}

\subsubsection{Definition of non-simple $\CLE$}

We are now going to review the definition of \clek{}.  We begin with the case $\kappa \in (4,8)$.  Suppose that $D \subseteq \C$ is a simply connected domain and $x \in \partial D$ is fixed.  Let $(y_n)$ be a fixed countable dense subset of $\partial D$ and, for each $n$, let $\eta_n$ be an $\SLE_\kappa(\kappa-6)$ process in $D$ from $x$ to $y_n$ with force point at $x^+$.  We assume that the $\eta_n$ are coupled together so as to agree until their target points are separated after which they evolve conditionally independently.  That this is possible follows from the target invariance of $\SLE_\kappa(\kappa-6)$ \cite{sw2005coordinate} mentioned just above.  The collection $(\eta_n)$ defines the so-called $\CLE_\kappa$ exploration tree \cite{s2009cle} which can be used to generate the loops of the $\CLE_\kappa$ that intersect $\partial D$.  The remaining loops of the $\CLE_\kappa$ are then defined by iterating this procedure in the complementary components of the union of these loops (or, for non-nested $\CLE_\kappa$, only the components that are outside these loops).

The loops of the $\CLE_\kappa$ that intersect $\partial D$ are defined from the $(\eta_n)$ as follows.  Fix $n \in \N$ and let $y$ be on the clockwise arc of $\partial D$ from $y_n$ to $x$ and let $\tau$ be the first time that $\eta_n$ disconnects $y_n$ from $y$.  We then let $\sigma$ be the largest time before $\tau$ that $\eta_n$ is in the clockwise segment of $\partial D$ from $y_n$ to $x$.  Then there exists a sequence $(y_{n_k})$ on the clockwise boundary segment from $y_n$ to $\eta_n(\sigma)$ converging to $\eta_n(\sigma)$.  A loop of the $\CLE_\kappa$ is defined by considering the concatenation of $\eta_{n_k}$ starting from $\sigma$ up until the first time it disconnects $y_{n_k}$ from $\eta_n(\sigma)$.  That the loops are continuous curves was proved in \cite{ms2016ig1} and that the law of the resulting ensemble of loops does not depend on its starting point $x$ was proved in \cite{s2009cle,ms2016ig3}.

\subsubsection{Definition of simple $\CLE$}

We are now going to review the definition of \clek{} in the case $\kappa \in (8/3,4]$. In this case, one needs to consider the generalised \slekr{\kappa-6} processes. That these processes are also target invariant is proved in \cite{sw2005coordinate,s2009cle}. As before, we fix some $x \in \partial D$, a countable dense subset $(y_n)$ of $\partial D$, and consider a \clek{} exploration tree $(\eta_n)$ whose branches are given by \slekrb{\kappa-6}{\beta} (resp.\ \slefmu{\mu} in the case $\kappa=4$) processes from $x$ targeting the points in $(y_n)$. The way in which the loops of the \clek{} are defined from the exploration tree is slightly different here. Let $\eta_n$ be a branch of the exploration tree and let $t$ be a time when $W_t \neq V_t$ in~\eqref{eq:slekr}. Let $\tau$ (resp.\ $\sigma$) be the next (resp.\ last) time when $W_t = V_t$. Then $\eta_n([\sigma,\tau])$ is a simple loop. A collection of \clek{} loops is defined by considering each branch of $(\eta_n)$ and each $t$. The remaining loops of the $\CLE_\kappa$ are then defined by iterating this procedure in the complementary components of the exploration tree (or, for non-nested $\CLE_\kappa$, only the components that are outside the \clek{} loops).

The fact that the loops are continuous curves and that the law of the resulting ensemble of loops does not depend on its starting point $x$ nor the parameter $\beta$ (resp.\ $\mu$) was proved in \cite{sw2012cle,ww-cle-explorations}.

There is an equivalent construction of simple CLE by the \emph{Brownian loop soup} shown in \cite{sw2012cle}. Let $\kappa \in (8/3,4]$ and $\mathbf{c} = (3\kappa-8)(6-\kappa)/(2\kappa) \in (0,1]$. The Brownian loop soup of intensity $\mathbf{c}$ in $D$ is defined as the Poisson point process with intensity $\mathbf{c}\mu_D$ where $\mu_D$ is the Brownian loop measure restricted to the set of loops within $D$ \cite{lw-loop-soup}. A cluster of a Brownian loop soup $\Lambda_D$ is an equivalence class in $\Lambda_D$ under the relation that relates all pairs of intersecting loops. Then the law of the collection of outer boundaries of the outermost loop soup clusters of $\Lambda_D$ is that of a non-nested \clek{} in $D$.

\subsubsection{Multichordal $\CLE_\kappa$}
\label{subsubsec:multichordal_cle}

We are now going to recall some of the results from \cite{amy-cle-resampling} about partially explored $\CLE_{\kappa}$, as they will be used extensively in the present article.  Let us begin by recalling \cite[Definition~1.1]{amy-cle-resampling}.
\begin{definition}
\label{def:marked_domain}
Let $D \subsetneq \C$ be a simply connected domain.  Fix $N \in \N_0$ and let $\ul{x} = (x_1,\ldots,x_{2N})$ be a collection of distinct prime ends which are ordered counterclockwise.
\begin{enumerate}[(i)]
\item We call $(D;\ul{x})$ a \emph{marked domain}. Let $\markeddomain{2N}$ be the collection of marked domains with $2N$ marked points.
\item Suppose that $(D;\ul{x}) \in \markeddomain{2N}$, and we have a non-crossing collection of paths $\ul{\gamma} = (\gamma_1,\ldots,\gamma_N)$ outside of $D$ where each $\gamma_i$ connects a distinct pair of marked points. The paths $\ul{\gamma}$ induce a planar link pattern $\beta$ on the exterior of $D$, formally described by a partition of the marked points into pairs $\beta = \{\{a_1,b_1\},\ldots,\{a_N,b_N\}\}$ where $\{a_1,b_1,\ldots,a_N,b_N\} = \{1,\ldots,2N\}$, $a_r<b_r$, and such that the configuration $a_r<a_s<b_r<b_s$ does not occur. We let $\eldomain{2N}$ be the collection of exterior link pattern decorated marked domains.
\item Suppose that $(D;\ul{x}) \in \markeddomain{2N}$, and we have a non-crossing collection of paths $\ul{\eta} = (\eta_1,\ldots,\eta_N)$ in~$D$ where each $\eta_i$ connects a distinct pair of marked points. We consider similarly the planar link pattern $\alpha$ on the interior of $D$ induced by the paths $\ul{\eta}$. We let $\ildomain{2N}$ be the collection of interior link pattern decorated marked domains.
\end{enumerate}
\end{definition}

We review the definition of the multichordal \clek{} in a marked domain $(D;\ul{x};\beta) \in \eldomain{2N}$ (resp.\ $(D;\ul{x};\alpha) \in \ildomain{2N}$) whose law we will denote by $\mcclelaw{D;\ul{x};\beta}$ (resp.\ $\mcclelaw{D;\ul{x};\alpha}$). A sample from $\mcclelaw{D;\ul{x};\beta}$ (resp.\ $\mcclelaw{D;\ul{x};\alpha}$) consists of a pair $(\ul{\eta}, \Gamma)$ where $\ul{\eta} = (\eta_1,\ldots,\eta_N)$ is a family of curves that connect the marked points $\ul{x}$ and, given $\ul{\eta}$, the collection $\Gamma$ consists of a conditionally independent nested $\CLE_{\kappa}$ in each of the complementary components of the~$\ul{\eta}$.

For $(D;\ul{x};\alpha) \in \ildomain{2N}$, the law of $\ul{\eta}$ under $\mcclelaw{D;\ul{x};\alpha}$ is that of a \emph{multichordal \slek{}} with interior link pattern given by $\alpha$. This is the unique law on curves that induce the link pattern $\alpha$ and such that for each~$i$, the conditional law of the curve starting at $x_i$ given the other curves is that of a chordal \slek{} in its complementary component.

For $(D;\ul{x};\beta) \in \eldomain{2N}$, the law of $\ul{\eta}$ under $\mcclelaw{D;\ul{x};\beta}$ is characterized as follows. In the case $N=1$ which we call a \emph{monochordal~\clek{}}, the law of $\eta_1$ is given by an \slek{} in $D$ from $x_1$ to $x_2$. In the case $N=2$ which we call a \emph{bichordal \clek{}}, we sample the interior link pattern $\alpha$ between the four points $x_1,\ldots,x_4$ from an explicit law determined in \cite{msw2020nonsimple,mw2018connection} (see \cite[equation~(1.1)]{amy-cle-resampling}), and given $\alpha$, the two chords $(\eta_1,\eta_2)$ are sampled from the law $\mcclelaw{D;\ul{x};\alpha}$. For general $N \ge 2$, the law of $\ul{\eta}$ under $\mcclelaw{D;\ul{x};\beta}$ is the unique law that is invariant under the following resampling kernels: Fix any $1 \leq i_1 < i_2 < i_3 < i_4 \leq 2N$ and condition on the $\eta_k$ that do not start at $x_{i_1},\ldots,x_{i_4}$. If there are exactly two chords connecting these four points and they are in the same complementary component of the remaining chords, resample the two chords according to the law of a bichordal \clek{} with exterior link pattern induced by $\beta$ and the remaining chords.

The following properties are shown in \cite[Theorems~1.3 and~1.6]{amy-cle-resampling}:
\begin{itemize}
 \item For each $(D;\ul{x};\beta) \in \eldomain{2N}$ and each interior link pattern $\alpha$ in $(D;\ul{x})$, the conditional law under $\mcclelaw{D;\ul{x};\beta}$ given the event that $\ul{\eta}$ induces the link pattern $\alpha$ is exactly $\mcclelaw{D;\ul{x};\alpha}$.
 \item The multichordal \clek{} law is conformally invariant, i.e., if $\varphi\colon D \to \wt{D}$ is a conformal transformation, then the law $\mcclelaw{\wt{D};\varphi(\ul{x});\beta}$ (resp.\ $\mcclelaw{\wt{D};\varphi(\ul{x});\alpha}$) is given by the pushforward of $\mcclelaw{D;\ul{x};\beta}$ (resp.\ $\mcclelaw{D;\ul{x};\alpha}$).
\end{itemize}

Next, we describe the \emph{partial explorations} of a $\CLE_{\kappa}$.  Suppose that $D \subseteq \C$ is a simply connected domain and let $\varphi \colon D \to \D$ a conformal transformation.  Let $\domainpair{D}$ consist of all pairs $(U,V)$ of simply connected domains $U \subseteq V \subseteq D$ with $\dist(\varphi(D \setminus V), \varphi(U)) > 0$.  (Note that $\domainpair{D}$ does not depend on the choice of $\varphi$.)  Suppose that $\Gamma$ is a nested $\CLE_{\kappa}$ in $D$ and $(U,V) \in \domainpair{D}$.   Let $\Gamma_\outside^{*,V,U}$ be the collection of maximal arcs of loops of $\Gamma$ that are disjoint from $U$ and intersect $D \setminus V$.  We call $\Gamma_\outside^{*,V,U}$ the \emph{partial exploration} of the loops of $\Gamma$ intersecting $D \setminus V$ up until they hit $U$. Let $V^{*,U}$ be the connected component containing $U$ after removing the loops and strands of $\Gamma_\outside^{*,V,U}$, and let $\Gamma_\inside^{*,V,U}$ denote the remainder of $\Gamma$ in $V^{*,U}$.

\begin{theorem}[{\cite[Theorem~1.12]{amy-cle-resampling}}]\label{thm:cle_partially_explored}
In the setup described above, the conditional law given $\Gamma_\outside^{*,V,U}$ of the remainder $\Gamma_\inside^{*,V,U}$ has the law of a multichordal $\CLE_{\kappa}$ in $V^{*,U}$ where the exterior link pattern $\beta$ is induced by the arcs of~$\Gamma_\outside^{*,V,U}$.
\end{theorem}

The following result will be useful.

\begin{proposition}[{\cite[Theorem~1.6(iv)]{amy-cle-resampling}}]\label{pr:link_probability}
Let $(D;\ul{x};\beta) \in \eldomain{2N}$ and $\alpha$ be an interior link pattern in $(D;\ul{x})$. Let $\alpha(\ul{\eta})$ be the interior link pattern induced by $\ul{\eta}$. Then $\mcclelaw{D;\ul{x};\beta}[\alpha(\ul{\eta}) = \alpha] > 0$, and this probability depends continuously on $\ul{x}$.
\end{proposition}

Finally, we collect some local absolute continuity results for \clek{} and multichordal \clek{}. For each compact set $K \subseteq D$, let $\Gamma|_K$ be the collection of non-trivial strands of $\Gamma$ restricted to $K$, i.e.\ the collection of non-trivial segments of $\CL \cap K$ for each $\CL \in \Gamma$.

\begin{lemma}\label{le:cle_abs_cont}
 Let $D_1,D_2$ be two simply connected domains, let $\Gamma_1,\Gamma_2$ be respectively a nested \clek{} in $D_1,D_2$, and let $\Upsilon_1,\Upsilon_2$ be their exterior gaskets. Let $K \subseteq D_1 \cap D_2$ be a compact set. Then the laws of $(\Gamma_1|_K, \Upsilon_1 \cap K)$ and of $(\Gamma_2|_K, \Upsilon_2 \cap K)$ are mutually absolutely continuous.
\end{lemma}

\begin{proof}
 In the case $\kappa \in (4,8)$, this follows from the same argument as in \cite[Lemma~7.2]{my2025geouniqueness}. In the case $\kappa \in (8/3,4]$, we are going to deduce it from the loop soup construction of \clek{} together with a resampling argument.

 First, it is explained in \cite[Section~5]{amy-cle-resampling}\footnote{Note that although the main results in \cite[Section~5]{amy-cle-resampling} are for $\kappa \in (4,8)$, the resampling procedure is equally valid for all $\kappa \in (8/3,8)$.} that as a consequence of Theorem~\ref{thm:cle_partially_explored} and Proposition~\ref{pr:link_probability}, we can resample the \clek{} in $D_1$ so that the following holds. Let $\Gamma_1^\resampled$ denote the resampled \clek{} in $D_1$. Then with positive conditional probability given $\Gamma_1$ we have that $(\Gamma_1|_K, \Upsilon_1 \cap K) = (\Gamma_1^\resampled|_K, \Upsilon_1^\resampled \cap K)$ and every loop of $\Gamma_1^\resampled$ whose interior intersects $K$ stays within $D_1 \cap D_2$. Let $F$ denote this event.

 We can then couple $\Gamma_1^\resampled$ and $\Gamma_2$ with the same Brownian loop soup $\Lambda$ so that for $i=1,2$, we have that $\Lambda_{D_i}$ is the collection of loops in $\Lambda$ contained in $D_i$. By the local finiteness of the Brownian loop measure, on the event $F$, the conditional probability given $\Gamma_1^\resampled$ is positive that the loops of $\Lambda_{D_2} \setminus \Lambda_{D_1 \cap D_2}$ are disjoint from the clusters whose interiors intersect $K$. Further, the conditional probability that they do not form new clusters that surround $K$ is also a.s.\ positive which can be seen from the local finiteness of \clek{} and a $0$-$1$-law argument. On this event we then have $(\Gamma_2|_K, \Upsilon_2 \cap K) = (\Gamma_1^\resampled|_K, \Upsilon_1^\resampled \cap K)$.
\end{proof}

The next lemma establishes the local absolute continuity between multichordal \clek{} and usual \clek{} in a quantitative way. To state it, we introduce the following definitions.

Suppose that $(D;\ul{x}) \in \markeddomain{2N}$. Let us consider the $2N$ boundary arcs of $\partial D$ between the marked points $\ul{x}$ which we can regard as alternatingly \emph{wired} and \emph{free}. Let $\CI$ be a non-empty collection of \emph{even} (or \emph{wired}) boundary arcs. For a multichordal \clek{} $(\ul{\eta},\Gamma)$, let $\Upsilon_\CI$ be the union of the gaskets of $\Gamma$ in the components of $D \setminus \ul{\eta}$ connected to $\CI$. For a usual nested \clek{}, we let $\Upsilon$ denote its exterior gasket.

For each compact set $K \subseteq D$, let $\Gamma|_K$ be the collection of non-trivial strands of $\Gamma$ and $\ul{\eta}$ restricted to $K$, not distinguishing between the strands of $\Gamma$ and $\ul{\eta}$.
Let $\CE_K$ be the collection of events that are measurable with respect to $(\Gamma|_K, \Upsilon_\CI \cap K)$. We write $E = E(\Upsilon_\CI)$ to emphasis the dependence on the choice of the gasket.

\begin{lemma}\label{le:mcle_abs_cont}
 Let $D \subseteq \C$ be a simply connected domain, $K \subseteq D$ a compact set, and $\varphi\colon \D \to D$ a conformal map. For each $\dsep>0$ and $\varepsilon>0$ there exists $\delta>0$ such that the following holds. Suppose that $(D;\ul{x};\beta) \in \eldomain{2N}$ and that $\abs{\varphi^{-1}(x_i)-\varphi^{-1}(x_j)} \ge \dsep$ for each $i \neq j$. Let $\CI$ be a non-empty collection of even boundary arcs as described above. Let $\mcclelaw{D;\ul{x};\beta}$ be the law of a multichordal \clek{} in $(D;\ul{x};\beta)$, and $\p_D$ the law of a nested \clek{} in $D$. Then, for each $E \in \CE_K$ with $\p_D[E(\Upsilon)] < \delta$, we have $\mcclelaw{D;\ul{x};\beta}[E(\Upsilon_\CI)] < \varepsilon$, and vice versa. The same holds for $\mcclelaw{D;\ul{x};\alpha}$ where $(D;\ul{x};\alpha) \in \ildomain{2N}$.
\end{lemma}

\begin{figure}[ht]
\centering
\includegraphics[width=0.4\textwidth,page=1]{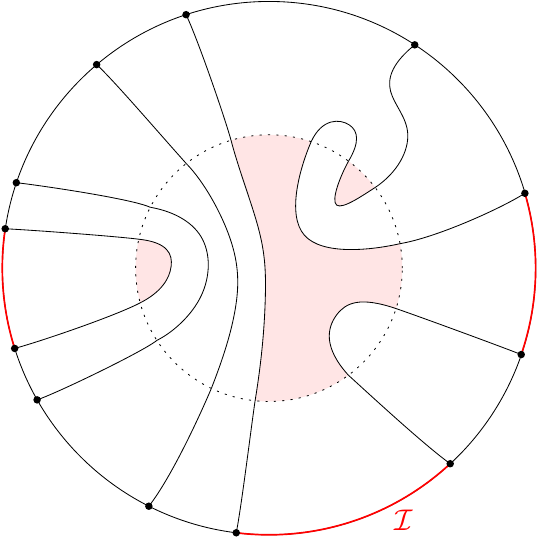}\hspace{0.1\textwidth}\includegraphics[width=0.4\textwidth,page=2]{resample_abs_cont.pdf}
\caption{Illustration of how the chords $\ul{\eta}$ can be resampled in $D \setminus K$ as to link outside $K$ while preserving the gasket $\Upsilon_\CI \cap K$ (indicated in red).}
\label{fi:resample_abs_cont}
\end{figure}

\begin{proof}
By \cite[Proposition~6.3]{amy-cle-resampling}, the law of $(\Gamma|_K, \Upsilon_\CI \cap K)$ under $\mcclelaw{D;\ul{x};\beta}$ (resp.\ $\mcclelaw{D;\ul{x};\alpha}$) is continuous in total variation in the Carath\'eodory topology of $(D;\ul{x}) \in \markeddomain{2N}$. Therefore it suffices to argue that for each particular $(D;\ul{x};\beta) \in \eldomain{2N}$ (resp.\ $(D;\ul{x};\alpha) \in \ildomain{2N}$), the laws of $(\Gamma|_K, \Upsilon_\CI \cap K)$ under $\mcclelaw{D;\ul{x};\beta}$, $\mcclelaw{D;\ul{x};\alpha}$, and $\p_D$ are mutually absolutely continuous; the uniform bound then follows by the local compactness of the Carath\'eodory topology.

Let $E \in \CE_K$ with $\mcclelaw{D;\ul{x};\beta}[E(\Upsilon_\CI)] > 0$. It is explained in \cite[Section~5]{amy-cle-resampling}\footnote{Note that although the main results in \cite[Section~5]{amy-cle-resampling} are for $\kappa \in (4,8)$, the resampling procedure is equally valid for all $\kappa \in (8/3,8)$.} that as a consequence of Theorem~\ref{thm:cle_partially_explored} and Proposition~\ref{pr:link_probability}, we can resample the multichordal \clek{} in $D \setminus K$ so that with positive conditional probability given $(\ul{\eta},\ul{\Gamma})$ the following hold: Let $(\ul{\eta}^\resampled,\Gamma^\resampled)$ denote the resampled multichordal \clek{}.
\begin{itemize}
 \item We have $\ul{\eta}^\resampled \subseteq D \setminus K$.
 \item The components of $\Upsilon_\CI \cap K$ remain connected to $\CI$; the other components of $K \setminus \ul{\eta}$ are inside loops of $\Gamma^\resampled$. (This is possible since $\CI$ contains boundary arcs of the same parity.)
\end{itemize}
See Figure~\ref{fi:resample_abs_cont} for an illustration. Let $D'$ be the connected component of $D \setminus \ul{\eta}^\resampled$ containing $K$, let $\Gamma^\resampled_{D'}$ be the loops of $\Gamma^\resampled$ in $D'$, and let $\Upsilon'$ be the exterior gasket of $\Gamma^\resampled_{D'}$. On the event described above we have $\Gamma|_K = \Gamma^\resampled_{D'}|_K$ and $\Upsilon_\CI \cap K = \Upsilon' \cap K$. This shows that $\p_{D'}[E(\Upsilon')] > 0$ for a random domain $D'$ with $K \subseteq D' \subseteq D$. By Lemma~\ref{le:cle_abs_cont}, we then also have $\p_D[E(\Upsilon)] > 0$.

Next, suppose that $E \in \CE_K$ and $\p_D[E(\Upsilon)] > 0$. By \cite[Proposition~6.3]{amy-cle-resampling}, we also have $\p_{D'}[E(\Upsilon)] > 0$ for all $D'$ sufficiently close to $D$. By \cite[Lemma~3.10]{amy-cle-resampling}, the $\mcclelaw{D;\ul{x};\alpha}$-probability that $\ul{\eta}$ stay in a sufficiently close neighborhood of $\partial D$ is positive. Moreover, if $I$ is a boundary arc in the non-empty collection $\CI$, then we can require $\ul{\eta}$ to stay in a small neighborhood of $\partial D \setminus I$. If this occurs and $D'$ is the connected component of $D \setminus \ul{\eta}$ containing $K$, then $\partial D' \cap I \neq \varnothing$, hence the exterior gasket $\Upsilon'$ of $\Gamma_{D'}$ (the loops of $\Gamma$ in $D'$) agrees with $\Upsilon_\CI$ in $D'$. Therefore $\mcclelaw{D;\ul{x};\alpha}[E(\Upsilon_\CI)] > 0$.

The final implication that $\mcclelaw{D;\ul{x};\alpha}[E(\Upsilon_\CI)] > 0$ implies $\mcclelaw{D;\ul{x};\beta}[E(\Upsilon_\CI)] > 0$ follows from Proposition~\ref{pr:link_probability}.
\end{proof}

\subsection{Geodesic and resistance metrics on $\CLE$ gaskets}
\label{subsec:cle_metrics}

We will now recall the definitions of the $\CLE_\kappa$ metrics.  The material here is only important for parts~\eqref{it:geo_count} and~\eqref{it:res_count} of Theorem~\ref{th:minkowski}.  Readers who are interested in only the other parts of Theorem~\ref{th:minkowski} can skip this section. For $\kappa \in (4,8)$, the canonical geodesic metric was constructed in \cite{my2025geouniqueness} and the canonical resistance metric in \cite{my2025resuniqueness}, both building off the general tightness result \cite{amy2025tightness}. We will review the definition in the case $\kappa \in (4,8)$. The geodesic metric for \clek{}, $\kappa \in (8/3,4)$, is defined in a similar way, see \cite{mt-geodesic} for details.

Let $\kappa \in (4,8)$, and let $\Gamma$ be a nested \clek{} in $D$. The \clek{} metrics are defined on each of the \emph{interior} gaskets of~$\Gamma$. Here, they are regarded as sets of \emph{prime ends} equipped with the topology given by $\dpathY$ in~\eqref{eq:dpath}, formally defined as the metric space completions of $(D \setminus \bigcup\Gamma, \dpathY)$. Note that the components of $\Upsilon_{\partial D} \setminus \partial D$ where $\Upsilon_{\partial D}$ is the exterior gasket are also considered as individual interior gaskets.

\subsubsection{Geodesic \clek{} metric}

We now define the \emph{geodesic \clek{} metric}. For a metric $\met{\cdot}{\cdot}{\Gamma}$ defined on a gasket $\Upsilon$ of $\Gamma$ and path $\gamma$ in $\Upsilon$ we let
\[
 \lmet{\gamma}=\sup_{t_0\leq\ldots\leq t_M}\sum_{i=1}^M \met{\gamma(t_{i-1})}{\gamma(t_{i})}{\Gamma}
\]
be its \emph{length} with respect to $\met{\cdot}{\cdot}{\Gamma}$. Recall that the metric $\met{\cdot}{\cdot}{\Gamma}$ is called a geodesic metric if for each $x,y \in \Upsilon$ there is a path $\gamma$ in $\Upsilon$ from $x$ to $y$ with $\lmet{\gamma} = \met{x}{y}{\Gamma}$.

If $U \subseteq \C$ is an open set, then $\Upsilon \cap U$ splits into a countable collection of connected components with respect to $\dpathY$. The \emph{internal metric} in $U$ is defined on each connected component of $\Upsilon \cap U$ and each gasket $\Upsilon$ of $\Gamma$ as follows.
\begin{equation}\label{eq:internal_metric}
 \metres{U}{x}{y}{\Gamma} = \inf_{\gamma \in \paths{x}{y}{U}{\Gamma}} \lmet{\gamma} ,\quad x,y \in \Upsilon \cap U 
\end{equation}
where $\paths{x}{y}{U}{\Gamma}$ is the set of admissible paths for $\Gamma$ between $x$ and $y$ contained in $U \cap D$ (as defined above~\eqref{eq:dpath}).

It is shown in \cite{amy2025tightness,my2025geouniqueness} that there is a unique (up to a deterministic constant) collection of metrics on the interior gaskets of the nested \clek{} for each simply connected domain $D \subseteq \C$ such that the following properties hold.

\textbf{Geodesic metric, continuity:} For each interior gasket $\Upsilon$ of $\Gamma$, the metric $\met{\cdot}{\cdot}{\Gamma}$ on~$\Upsilon$ is a geodesic metric, and it is continuous with respect to~$\dpathY$.

\textbf{Locally determined:} For each open $U \subseteq \C$, the collection of internal metrics $\metres{U}{\cdot}{\cdot}{\Gamma}$ defined in~\eqref{eq:internal_metric} is almost surely a measurable function of $\Gamma|_U$ that does not depend on the choice of $D$. Here $\Gamma|_U$ denotes the countable collection of loops and strands of $\CL \cap U$ for each $\CL \in \Gamma$.

\textbf{Translation invariance:} For each simply connected domain $D \subseteq \C$ and each $z \in \C$, we have $\met{\cdot+z}{\cdot+z}{\Gamma+z} = \met{\cdot}{\cdot}{\Gamma}$ almost surely.

Further, the geodesic \clek{} metric is \emph{conformally covariant} in the following sense. There is a constant $\geoexp > 0$ depending only on $\kappa$ such that if $\varphi\colon D \to \wt{D}$ is a conformal transformation, then almost surely, for each admissible path $\gamma \subseteq D$,
 \[
  \len{\met{\cdot}{\cdot}{\varphi(\Gamma)}}{\varphi(\gamma)} = \int \abs{\varphi'(\gamma(t))}^{\geoexp} \,d\len{\met{\cdot}{\cdot}{\Gamma}}{\gamma} .
 \]

\subsubsection{Resistance metric}

We now define the \emph{\clek{} resistance metric}. First, we recall the definition of a resistance metric \cite{k2001analysis}. A metric $R$ on a non-empty set $X$ is called a \emph{resistance metric} if for each finite subset $A \subseteq X$ there is a (unique) symmetric weight function $w\colon A \times A \to [0,\infty)$ such that the restriction of $R$ to $A$ agrees with the effective resistance metric associated with the weighted graph $G = (A,w)$.

Suppose that $\Upsilon$ is an interior gasket of $\Gamma$ and $\rmet{\cdot}{\cdot}{\Gamma}$ is a resistance metric on $\Upsilon$. We define the collection of \emph{internal metrics} for the resistance metric. For each finite subset $A \subseteq \Upsilon$ we let $w_A\colon A \times A \to [0,\infty)$ be the associated weight function. Suppose that $w_A(x,y) = 0$ whenever $x,y$ are separated by $A \setminus \{x,y\}$ in $\Upsilon$. Let $v_1,\ldots,v_n \in \Upsilon$, and let $V$ be a connected component of $\Upsilon$ with respect to $\dpathY$. Then (see \cite[Section~6.2]{my2025resuniqueness}) there is a unique resistance metric $\rmetres{V}{\cdot}{\cdot}{\Gamma}$ on~$\ol{V}$ such that the following holds. Let $A \subseteq \Upsilon$ be a finite subset containing $v_1,\ldots,v_n$ and such that each pair $v_i,v_j$ is disconnected by $A \setminus \{v_1,\ldots,v_n\}$ in~$\Upsilon$. Let $w^V_A$ be the weight function associated with the restriction of $\rmetres{V}{\cdot}{\cdot}{\Gamma}$ to $A \cap \ol{V}$. Then
\begin{equation}\label{eq:internal_rmet}
 w^V_A(x,y) = w_A(x,y) \quad\text{for each } A \cap \ol{V} .
\end{equation}

It is shown in \cite{amy2025tightness,my2025resuniqueness} that there is a unique (up to a deterministic constant) collection of metrics on the interior gaskets of the nested \clek{} for each simply connected domain $D \subseteq \C$ such that the following properties hold.

\textbf{Resistance metric, continuity, locality:} For each interior gasket $\Upsilon$ of $\Gamma$, the metric $\rmet{\cdot}{\cdot}{\Gamma}$ on~$\Upsilon$ is a resistance metric, and it is continuous with respect to~$\dpathY$. Furthermore, if $A \subseteq \Upsilon$ is a finite subset and $w_A\colon A \times A \to [0,\infty)$ is the associated weight function, then $w_A(x,y) = 0$ whenever $x,y$ are separated by $A \setminus \{x,y\}$ in $\Upsilon$.

\textbf{Locally determined:} For each open $U \subseteq \C$, the collection of internal metrics $\rmetres{V}{\cdot}{\cdot}{\Gamma}$ for each $\ol{V} \subseteq U$ defined above~\eqref{eq:internal_rmet} is almost surely a measurable function of $\Gamma|_U$ that does not depend on the choice of $D$. Here $\Gamma|_U$ denotes the countable collection of loops and strands of $\CL \cap U$ for each $\CL \in \Gamma$.

\textbf{Translation invariance:} For each simply connected domain $D \subseteq \C$ and each $z \in \C$, we have $\rmet{\cdot+z}{\cdot+z}{\Gamma+z} = \rmet{\cdot}{\cdot}{\Gamma}$ almost surely.

Further, the \clek{} resistance metric is \emph{scaling covariant} in the following sense. There is a constant $\resexp > 0$ depending only on $\kappa$ such that for each simply connected domain $D$ and each $\lambda > 0$ we have
\[
\rmet{\lambda\cdot}{\lambda\cdot}{\lambda\Gamma} = \lambda^{\resexp}\rmet{\cdot}{\cdot}{\Gamma} .
\]

\subsection{Independence across scales}
\label{subsec:independence_across_scales}

The proofs in this paper make use of the independence across scales argument established in~\cite{amy-cle-resampling}. For this, we set some notation.

Recall the definition of $\CS_j$ in Section~\ref{subsec:main_results}. For each $j \in \N$ and $Q \in \CS_j$ we let $A_Q = 2Q \setminus (3/2)Q$. Let $A_Q^\inside = (3/2)Q$ (resp.\ $A_Q^\outside = \C \setminus 2Q$) denote the inner (resp.\ outer) face of $A_Q$. Note that if $Q,Q' \in \CS_j$ do not share a vertex, their corresponding annuli $A_Q,A_{Q'}$ are disjoint. Further, if $\wt{Q} \in \CS_k$ for some $k>j$ and $\wt{Q} \subseteq Q$, then $A_{\wt{Q}} \subseteq A_Q^\inside$.

Let $\Gamma$ be a nested \clek{} in $D$. For each $z \in D$ and $j \in \N$ sufficiently large, there is a unique square $Q_{z,j} \in \CS_j$ containing $z$. Let $B_{z,j}^\inside = A_{Q_{z,j}}^\inside$ and $B_{z,j}^\outside = \C \setminus A_{Q_{z,j}}^\outside = (4/3)B_{z,j}^\inside$. We let $\CF_{z,j}$ be the $\sigma$-algebra generated by the partial exploration $\Gamma_\outside^{*,B_{z,j}^\outside,B_{z,j}^\inside}$ as defined in Section~\ref{subsubsec:multichordal_cle}. Further, we let $\CF_{z,j,\alpha}$ be the $\sigma$-algebra generated by $\CF_{z,j}$ and the interior link pattern of $\Gamma_\inside^{*,B_{z,j}^\outside,B_{z,j}^\inside}$.

By Theorem~\ref{thm:cle_partially_explored}, the conditional law of the unexplored part $\Gamma_\inside^{*,B_{z,j}^\outside,B_{z,j}^\inside}$ given $\CF_{z,j}$ is that of a multichordal \clek{} in $(B_{z,j}^\outside)^{*,B_{z,j}^\inside}$ where the exterior link pattern is induced by $\Gamma_\outside^{*,B_{z,j}^\outside,B_{z,j}^\inside}$. Note that the exterior link pattern \emph{does} depend on the configuration of $\Gamma$ outside $B_{z,j}^\outside$. But if we condition on $\CF_{z,j,\alpha}$, then the conditional law of $\Gamma_\inside^{*,B_{z,j}^\outside,B_{z,j}^\inside}$ depends only on $(B_{z,j}^\outside)^{*,B_{z,j}^\inside}$, its marked boundary points, and the interior link pattern.

The results in~\cite{amy-cle-resampling} tell us that the conditional probabilities given $\CF_{z,j}$ can be uniformly controlled when the strands are ``separated''. We now give the definition of separation and recall the key lemma giving us the independence across scales property.

Let $\dsep > 0$ be some small number. Let $\Esep_{z,j} = \Esep_{z,j}(\dsep)$ denote the event that the marked points on the boundary of $(B_{z,j}^\outside)^{*,B_{z,j}^\inside}$ have distance at least $\dsep 2^{-j}$ from each other. We state the following variant of \cite[Lemma~4.4]{amy-cle-resampling} which follows from the exact same proof.

\begin{lemma}
\label{le:separation_event}
For any $q \in (0,1)$ and $b>1$ there exist $\dsep>0$ and $c>0$ such that the following is true. Let $z\in D$, $j_0 \in \N$ such that $B(z,2^{-j_0}) \subseteq D$. Then for each $k \in \N$ the probability that more than $q$ fraction of the events $(\Esep_{z,j})^c$ where $j = j_0+1,\ldots,j_0+k$ occur is at most $ce^{-bk}$.
\end{lemma}

\section{Moments of the gasket measure}
\label{sec:moment_finiteness}

The purpose of this section is to prove the Propositions~\ref{pr:measure_ub} and~\ref{pr:measure_lb} which give us sharp bounds for the volume measure on the $\CLE_{\kappa}$ gasket in a small ball. As in Section~\ref{subsec:main_results}, we let $\kappa \in (8/3,8)$ be fixed and let $\dcle$ be as in~\eqref{eqn:cle_dim}. Let $\Gamma$ be a nested \clek{} in $D$. For each $\delta > 0$, we let $\Upsilon_{\ge\delta}$ be the union of the gaskets of $\Gamma$ with diameter at least $\delta$. Let $\meas{\cdot}{\Upsilon_{\ge\delta}}$ denote the sum of the volume measures from \cite{ms2022clemeasure} on each gasket of diameter at least $\delta$.

\subsection{Upper bound}

We start by proving the upper bound. The upper bound in Proposition~\ref{pr:measure_ub} is a consequence of the following result.

\begin{proposition}
\label{prop:cle_measure_moments}
Suppose that $\Gamma$ is a $\CLE_{\kappa}$ on a simply connected domain $D \subseteq \C$, and let $\meas{\cdot}{\Upsilon_{\ge\delta}}$ be the volume measure on the union of its gaskets with diameter at least $\delta$. For every $\delta > 0$, $n \in \N$, $a,b > 0$, and $K \subseteq D$ compact, there is a constant $c>0$ and events $E_M$ with $\p[E_M] \ge 1-cM^{-b}$ for each $M > 1$ such that, for all $z \in K$ and $r > 0$ so that $B(z,2r) \subseteq D$, we have that
\[ \E[ \meas{B(z,r)}{\Upsilon_{\ge\delta}}^n \one_{E_M}] \leq M r^{2+(n-1)\dcle -a} . \]
\end{proposition}

To start to prove Proposition~\ref{prop:cle_measure_moments}, let us first give the first moment bound for the outermost gasket.
\begin{lemma}
\label{lem:gasket_measure_first_moment}
There exists a constant $c > 0$ so that the following is true.  Suppose that $\Gamma$ is a $\CLE_{\kappa}$ on a simply connected domain $D \subseteq \C$, and let $\Upsilon$ be its outermost gasket.  Then for every $z \in D$ and $r > 0$ so that $B(z,2r) \subseteq D$ we have that
\begin{equation}\label{eq:gasket_measure_first_moment}
 \E[ \meas{B(z,r)}{\Upsilon} ] \leq c r^2 (\dist(z,\partial D))^{\dcle-2}.
\end{equation}
\end{lemma}

We emphasize that in the statement of Lemma~\ref{lem:gasket_measure_first_moment} the constant $c$ does not depend on $D$, $z$, or $r$.

\begin{proof}[Proof of Lemma~\ref{lem:gasket_measure_first_moment}]
Suppose that $\wt{\Gamma}$ is a $\CLE_{\kappa}$ on $\D$ and $\meas{\cdot}{\wt{\Upsilon}}$ is the corresponding measure on its gasket $\wt{\Upsilon}$.  Then it is shown in \cite[Lemma~5.1]{ms2022clemeasure} that the measure $\E[ \meas{\cdot}{\wt{\Upsilon}}]$ has a smooth density with respect to Lebesgue measure.  In particular, there exists a constant $c_1 > 0$ so that
\begin{equation}
\label{eqn:measure_b_0_s_bound}
\E[ \meas{B(0,s)}{\wt{\Upsilon}} ] \leq c_1 s^2.
\end{equation}
We are going to deduce~\eqref{eq:gasket_measure_first_moment} from~\eqref{eqn:measure_b_0_s_bound} and conformal covariance.

Let $\varphi \colon \D \to D$ be the unique conformal transformation with $\varphi(0) = z$ and $\varphi'(0) > 0$.  By the distortion estimates for conformal maps, we have that
\begin{equation}
\label{eqn:varphi_deriv_bound}
|\varphi'(w)| \asymp |\varphi'(0)| \asymp \dist(z,\partial D) \quad\text{for all}\quad w \in B(0,1/2) .
\end{equation}
Further, there exists $c_0 > 0$ such that if $r < \frac{1}{2}\dist(z,\partial D)$ and $s = r / \dist(z, \partial D)$ we have that $B(z,c_0 r) \subseteq \varphi(B(0,s))$.

Let $\Gamma$ be the \clek{} in $D$ given by $\Gamma = \varphi(\wt{\Gamma})$.  Then we have by \cite[Theorem~1.2]{ms2022clemeasure} that $d\meas{\cdot}{\Upsilon} = |\varphi'(\cdot)|^{\dcle} d \meas{\cdot}{\wt{\Upsilon}} \circ \varphi^{-1}$, consequently
\begin{equation}
\label{eqn:mu_b_z_r_ubd}
\meas{B(z,c_0 r)}{\Upsilon} \leq \meas{\varphi(B(0,s))}{\Upsilon} = \int_{B(0,s)} |\varphi'(w)|^{\dcle} \meas{dw}{\wt{\Upsilon}}.
\end{equation}
Inserting~\eqref{eqn:varphi_deriv_bound} into~\eqref{eqn:mu_b_z_r_ubd} and then taking an expectation and applying~\eqref{eqn:measure_b_0_s_bound} we see for a constant $c_2 > 0$ that
\begin{equation}
\label{eqn:expectation_measure_first_bound}
\E[ \meas{B(z,c_0 r)}{\Upsilon}] \leq c_2 \left( \frac{r}{\dist(z,\partial D)} \right)^2 \dist(z, \partial D)^{\dcle} = c_2 r^2 (\dist(z,\partial D))^{\dcle-2},
\end{equation}
which completes the proof.
\end{proof}

We will also need the result from \cite{ssw2009radii} which gives the exponent for the probability that the $\CLE_{\kappa}$ gasket intersects a given ball.

\begin{lemma}[{\cite[Theorem~2]{ssw2009radii}}]
\label{lem:ssw_gasket_exponent}
There exists a constant $c > 0$ so that the following is true.  Suppose that $\Gamma$ is a $\CLE_{\kappa}$ on a simply connected domain $D \subseteq \C$, and $\Upsilon$ its outermost gasket.  Let $z \in D$ and $r > 0$ so that $B(z,2r) \subseteq D$.  Then
\[ \p[ \Upsilon \cap B(z,r) \neq \emptyset ] \le c\left(\frac{r}{\dist(z,\partial D)}\right)^{2-\dcle} . \]
\end{lemma}

We now generalize the Lemmas~\ref{lem:gasket_measure_first_moment} and~\ref{lem:ssw_gasket_exponent} to the case of a multichordal \clek{} (see Section~\ref{subsubsec:multichordal_cle}). Suppose in the following that $(D;\ul{x};\alpha) \in \ildomain{2N}$, let $(\ul{\eta},\Gamma)$ be sampled from the law~$\mcclelaw{D;\ul{x};\alpha}$, i.e.\ the multichordal \clek{} in $(D;\ul{x})$ conditionally on the interior link pattern $\alpha$. Let $\Upsilon$ be the set of points that are connected to $\partial D$ by a path not crossing any loop of $\Gamma$, in other words the union of the outermost gaskets for each of the \clek{} in the complementary connected components of $\ul{\eta}$.

\begin{lemma}
\label{lem:gasket_measure_first_moment_mcle}
For each $N \in \N$, and $\dsep > 0$ there exists a constant $c > 0$ so that the following is true.  Suppose that $(\ul{\eta},\Gamma)$ is sampled from~$\mcclelaw{D;\ul{x};\alpha}$ where $(D;\ul{x};\alpha) \in \ildomain{2N}$. Let $z \in D$ and $r > 0$ so that $B(z,2r) \subseteq D$, and suppose that the harmonic measure in $D$ seen from $z$ of each of the boundary arcs between the marked points is at least $\dsep$.  Then
\[
 \mccleexp{D;\ul{x};\alpha}[ \meas{B(z,r)}{\Upsilon} ] \leq c r^2 (\dist(z,\partial D))^{\dcle-2}.
\]
\end{lemma}

\begin{proof}
As in the proof of Lemma~\ref{lem:gasket_measure_first_moment}, we can assume $D = \D$, $z = 0$ due to the conformal covariance of the gasket measure. By symmetry, it suffices to consider the gaskets $\Upsilon_{\mathrm{even}}$ that are connected to the even boundary arcs.

Consider the multichordal \clek{} law~$\mcclelaw{\D;\ul{x};\beta_0}$ where $\beta_0$ is the \emph{exterior} link pattern given by $\beta_0 = \{\{1,2\},\{3,4\},\ldots\{2N-1,2N\}\}$. By Proposition~\ref{pr:link_probability}, the probability under~$\mcclelaw{\D;\ul{x};\beta_0}$ that $\ul{\eta}$ assumes any given interior link pattern $\alpha$ is bounded below by a positive constant depending only on $\dsep$. Therefore it suffices to prove the result for~$\mcclelaw{\D;\ul{x};\beta_0}$.

By \cite[Proposition~1.10]{amy-cle-resampling} we can obtain a multichordal \clek{} law from the remainder of explorations of a \clek{} from the boundary. Concretely, suppose that $\wt{\Gamma}$ is a \clek{} in $\D$. For each $\dsep > 0$, by \cite[Proposition~1.10, Lemma~3.10, and Proposition~3.15]{amy-cle-resampling} there is a constant $p_0 > 0$ such that the following is true. For each $(\D;\ul{x}) \in \markeddomain{2N}$ where $\min_{i \neq j} \abs{x_i-x_j} \ge \dsep$ there is a Markovian exploration of $\wt{\Gamma}$ that stops before any of the strands hit $B(0,1/2)$ and such that the probability that the resulting marked domain is conformally equivalent to $(\D;\ul{x};\beta_0)$ is at least $p_0$. Combining this with Lemma~\ref{lem:gasket_measure_first_moment} shows that
\[ \mccleexp{\D;\ul{x};\beta_0}[ \meas{B(z,r)}{\Upsilon_{\mathrm{even}}} ] \le p_0^{-1}\E_{\D}[ c\meas{B(z,2r)}{\wt{\Upsilon}} ] \le cp_0^{-1}r^2 . \]
\end{proof}

\begin{lemma}
\label{lem:ssw_gasket_exponent_mcle}
For each $N \in \N$, $\dsep > 0$, and $a>0$ there exists a constant $c > 0$ so that the following is true. Suppose that $(\ul{\eta},\Gamma)$ is sampled from~$\mcclelaw{D;\ul{x};\alpha}$ where $(D;\ul{x};\alpha) \in \ildomain{2N}$. Let $z \in D$ and $r > 0$ so that $B(z,2r) \subseteq D$, and suppose that the harmonic measure in $D$ seen from $z$ of each of the boundary arcs between the marked points is at least $\dsep$.  Then
\[ \mcclelaw{D;\ul{x};\alpha}[ \Upsilon \cap B(z,r) \neq \emptyset ] \le c\left(\frac{r}{\dist(z,\partial D)}\right)^{2-\dcle} . \]
\end{lemma}

\begin{proof}
This follows from Lemma~\ref{lem:ssw_gasket_exponent} by the exact same argument used to deduce Lemma~\ref{lem:gasket_measure_first_moment_mcle} from Lemma~\ref{lem:gasket_measure_first_moment}.
\end{proof}

\begin{lemma}\label{le:mu_no_atoms}
 Suppose that $\Gamma$ is a $\CLE_{\kappa}$ on a simply connected domain $D \subseteq \C$, let $\Upsilon$ be its outermost gasket, and let $\meas{\cdot}{\Upsilon}$ be the volume measure on $\Upsilon$. Then almost surely, for each $x \in \Upsilon$ we have $\meas{\{x\}}{\Upsilon} = 0$.
\end{lemma}

\begin{proof}
 Fix $K \subseteq D$ compact and $\varepsilon > 0$, and consider the event $G^\bad_{K,\varepsilon}$ that $\meas{\{x\}}{\Upsilon} \ge \varepsilon$ for some $x \in K$. We aim to show that $\p[G^\bad_{K,\varepsilon}] = 0$. For each $\delta > 0$, consider
 \[ X_{\varepsilon,\delta} = \sum_{z \in \delta\Z^2 \cap K} \meas{B(z,\delta)}{\Upsilon} \one_{\meas{B(z,\delta)}{\Upsilon} \ge \varepsilon} . \]
 Note that $X_{\varepsilon,\delta} \ge \varepsilon$ on $G^\bad_{K,\varepsilon}$, hence $\E[X_{\varepsilon,\delta}] \ge \varepsilon\p[G^\bad_{K,\varepsilon}]$. We will show that $\E[X_{\varepsilon,\delta}] \to 0$ as $\delta \searrow 0$, and for this we will show that
 \[
  \E[ \meas{B(z,\delta)}{\Upsilon} \one_{\meas{B(z,\delta)}{\Upsilon} \ge \varepsilon} ] = o(\delta^2)
  \quad\text{as } \delta \searrow 0
 \]
 with a rate that is uniform in $z \in K$.

For each $z \in K$ let $\eta_z$ be an exploration path of the \clek{} targeting $z$. Let $\tau_{z,\delta}$ be the first time when it hits $\partial B(z,2\delta)$. Then $\Upsilon \cap B(z,2\delta) \neq \emptyset$ if and only if $\eta|_{[0,\tau_{z,\delta}]}$ has not surrounded $z$ with a (true) \clek{} loop. (When $\kappa \in (4,8)$ and $\eta_z$ is the exploration path that traces the outermost loops of $\Gamma$ clockwise, say, this means that $\eta|_{[0,\tau_{z,\delta}]}$ has not made a clockwise loop around $z$. When $\kappa \in (8/3,4]$ and $\eta_z$ is an \slekrb{\kappa-6}{\beta} exploration, this means that the loops that $\eta|_{[0,\tau_{z,\delta}]}$ makes away from its trunk do not surround $z$.) Hence,
\begin{align}
 \p{\left[ \eta|_{[0,\tau_{z,\delta}]} \text{ has not surrounded $z$ with a \clek{} loop} \right]} = \p[ \Upsilon \cap B(z,2\delta) \neq \emptyset ] \lesssim \delta^{2-\dcle} \label{eqn:cle_not_surrounded}
\end{align}
 by Lemma~\ref{lem:ssw_gasket_exponent}. The conditional law of $\Gamma$ given $\eta|_{[0,\tau_{z,\delta}]}$ is that of a monochordal \clek{}. If the two marked points are sufficiently separated, we get by conformal covariance and Koebe's distortion theorem that
\begin{align}
  & \E{\left[ \meas{B(z,\delta)}{\Upsilon} \one_{\meas{B(z,\delta)}{\Upsilon} \ge \varepsilon} \mmiddle| \eta|_{[0,\tau_{z,\delta}]} \right]} \nonumber\\
  &\le \delta^{\dcle}\mccleexp{\D;-i,i}{\left[ \meas{B(0,1-c_0)}{\Upsilon} \one_{\meas{B(0,1-c_0)}{\Upsilon} \ge c_0\varepsilon\delta^{-\dcle}} \right]} \label{eqn:expected_mass_hit_far}
\end{align}
 where $c_0>0$ is a suitable constant. If the two marked points are very close, the same argument used in the proof of Lemma~\ref{lem:gasket_measure_first_moment_mcle} shows that
\begin{align}
  &\E{\left[ \meas{B(z,\delta)}{\Upsilon} \one_{\meas{B(z,\delta)}{\Upsilon} \ge \varepsilon} \mmiddle| \eta|_{[0,\tau_{z,\delta}]} \right]} \nonumber\\
  &\lesssim \delta^{\dcle}\E_{\D}{\left[ \meas{B(0,1-c_0)}{\Upsilon} \one_{\meas{B(0,1-c_0)}{\Upsilon} \ge c_0\varepsilon\delta^{-\dcle}} \right]} . \label{eqn:expected_mass_hit_close}
\end{align}
 In any case, since the measure has finite first moment by Lemma~\ref{lem:gasket_measure_first_moment_mcle}, the expectation on the right-hand side is $o(1)$ as $\delta \searrow 0$. Combining~\eqref{eqn:cle_not_surrounded}, \eqref{eqn:expected_mass_hit_far}, and~\eqref{eqn:expected_mass_hit_close} completes the proof.
\end{proof}

Next, we recall a general lemma that provides an efficient way to cover a collection of points by annuli which is used to bound the probability that a random fractal gets close to a given collection of points \cite[Lemma~A.2]{kms2023nonsimpleremovability}.

\begin{lemma}[{\cite[Lemma~A.2]{kms2023nonsimpleremovability}}]
\label{lem:annulus_lemma}
Fix $n \in \N$, $r_0 > 0$, and $K \subseteq \h$ compact.  There exists a constant $C = C(n,K,r_0)$ so that the following is true.  Suppose that $z_1,\ldots,z_n \in K$ are distinct.  For each $1 \leq i \leq n$ there exists $n_i \in \N$ and $s_{i,k},r_{i,k} \in [0,r_0)$ such that $0 < s_{i,k} < r_{i,k}$ for $1 \leq k \leq n_i-1$ and $0 = s_{i,n_i} < r_{i,n_i}$ so that the annuli $A(z_i, s_{i,k}, r_{i,k})$ are pairwise disjoint and
\begin{equation}\label{eq:annulus_cover_bound}
 \prod_{i=1}^n \frac{\prod_{k=1}^{n_i-1} s_{i,k}}{\prod_{k=1}^{n_i} r_{i,k}} \leq C \prod_{i=2}^n \frac{1}{\min_{i' < i} |z_i - z_{i'}|}
\end{equation}
with the convention $\prod_{k=1}^{n_i-1} s_{i,k} = 1$ if $n_i = 1$.
\end{lemma}

We will now use Lemmas~\ref{lem:gasket_measure_first_moment}--\ref{lem:annulus_lemma} to prove the main estimate for the proof of Proposition~\ref{prop:cle_measure_moments}.

\begin{lemma}
\label{lem:measure_has_density}
Suppose that $\Gamma$ is a $\CLE_{\kappa}$ on a simply connected domain $D \subseteq \C$, and let $\meas{\cdot}{\Upsilon_{\ge\delta}}$ be the volume measure on the union of its gaskets with diameter at least $\delta$. For every $\delta > 0$, $n \in \N$, $a,b>0$, and $K \subseteq D$ compact, there is a constant $c>0$ and events $E_M$ with $\p[E_M] \ge 1-cM^{-b}$ for each $M > 1$ such that the measure $\E[ (\meas{\cdot}{\Upsilon_{\ge\delta}} \otimes \cdots \otimes \meas{\cdot}{\Upsilon_{\ge\delta}}) \one_{E_M}]$ has a density $g$ with respect to Lebesgue measure that satisfies the bound
\begin{equation}
\label{eqn:cle_measure_density_ubd}
g(z_1,\ldots,z_n) \leq M \bigl( \min_{i \neq i'} |z_i - z_{i'}|\bigr)^{-a} \prod_{i=2}^n \bigl( \min_{i' < i} |z_i - z_{i'}| \bigr)^{\dcle - 2} ,
\quad z_1,\ldots,z_n \in K .
\end{equation}
\end{lemma}

\begin{proof}
By the conformal covariance of the gasket measure, it suffices to assume $D = \h$. Our aim is to define events $E_M$ with $\p[E_M] \geq 1-O(M^{-b})$ so that the following is true. Suppose that $z_1,\ldots,z_n \in K$ are distinct and that $r > 0$ is sufficiently small. Then
\begin{equation}
\label{eqn:ball_mass_ubd}
\E{\left[ \prod_{i=1}^n \meas{B(z_i,r)}{\Upsilon_{\ge\delta}} \, \one_{E_M} \right]} \leq M r^{2n} \bigl( \min_{i \neq i'} |z_i - z_{i'}|\bigr)^{-a} \prod_{i=2}^n \bigl( \min_{i' < i} |z_i - z_{i'}| \bigr)^{\dcle - 2}.
\end{equation}
From this we conclude that the measure $\E[ (\meas{\cdot}{\Upsilon_{\ge\delta}} \otimes \cdots \otimes \meas{\cdot}{\Upsilon_{\ge\delta}}) \one_{E_M}]$, at least away from the diagonals, is absolutely continuous with respect to the $n$-dimensional Lebesgue measure, and Lebesgue's differentiation theorem yields the upper bound~\eqref{eqn:cle_measure_density_ubd} for the density. Finally, since by Lemma~\ref{le:mu_no_atoms}, $\meas{\cdot}{\Upsilon_{\ge\delta}}$ a.s.\ does not have atoms, we obtain that $\E[ (\meas{\cdot}{\Upsilon_{\ge\delta}} \otimes \cdots \otimes \meas{\cdot}{\Upsilon_{\ge\delta}}) \one_{E_M}]$ is $0$ on the diagonals. This will conclude the proof of the lemma.

Let $\Esep_{z,j} = \Esep_{z,j}(\dsep)$ be the events defined in Section~\ref{subsec:independence_across_scales}, and let $\CF_{z,j,\alpha}$ be as defined in Section~\ref{subsec:independence_across_scales}. Fix $a>0$ small. For $j \in \N$, $2^{-j} < \dist(K,\partial D)$, let $\wt{E}_j$ be the event that for each $z \in K$ there is $(1-a)j \le j' < j$ such that $\Esep_{z,j'}$ occurs. Let $E_{j_0} = \cap_{j \ge j_0} \wt{E}_j$. Given any $b>0$, by Lemma~\ref{le:separation_event} there is $\dsep > 0$ and $c>0$ such that $\p[(E_{j_0})^c] \le ce^{-bj_0}$.

Suppose that we are on the event $E_{j_0}$. Let $r_0 = 2^{-j_0}$, and assume that $r_0 < \delta$. Given $z_1,\ldots,z_n$, let $r_{i,k}$ (resp.\ $s_{i,k}$) be the outer (resp.\ inner) radii for the annuli centered at the points $z_1,\ldots,z_n$ from Lemma~\ref{lem:annulus_lemma}. We explore the \clek{} from the outside to the inside. For each of the annuli $A(z_i,s_{i,k},r_{i,k})$, we find the first scale $j'$ within $B(z_i,r_{i,k})$ so that $\Esep_{z_i,j'}$ occurs. Note that each gasket of diameter at least $\delta$ intersecting $B(z_i,s_{i,k})$ crosses the annulus since $r_{i,k} < r_0 < \delta$. For each of the annuli $A(z_i,s_{i,k},r_{i,k})$ with $k < n_i$ we get from Lemma~\ref{lem:ssw_gasket_exponent_mcle} that
\[
 \p[ \Upsilon_{\ge\delta} \cap B(z_i,s_{i,k}) \neq \emptyset \mid \CF_{z_i,j',\alpha} ] \one_{\Esep_{z_i,j'}} \le c\left(\frac{s_{i,k}}{2^{-j'}}\right)^{2-\dcle} \le c\left(\frac{s_{i,k}}{r_{i,k}}\right)^{2-\dcle}r_{i,k}^{-a} ,
\]
and for each $B(z_i,r_{i,n_i})$ we get from Lemma~\ref{lem:gasket_measure_first_moment_mcle} that
\[
 \E[ \meas{B(z_i,r)}{\Upsilon_{\ge\delta}} \mid \CF_{z_i,j',\alpha} ] \one_{\Esep_{z_i,j'}} \le c r^2 2^{-j'(\dcle-2)} \le c\frac{r^2}{r_{i,n_i}^{2-\dcle+a}}
\]
Noting that we can choose the annuli so that the total number of annuli is at most $\sum_i n_i \le 2n$, we conclude by~\eqref{eq:annulus_cover_bound} that for some constants $c_1,c_2 > 0$ that
\begin{equation*}
\begin{split}
\E{\left[ \prod_{i=1}^n \meas{B(z_i,r)}{\Upsilon_{\ge\delta}} \, \one_{E_{j_0}} \right]}
&\leq c_1 r^{2n} (\min_i r_{i,n_i})^{-2na} \left(\prod_{i=1}^n \frac{\prod_{k=1}^{n_i-1} s_{i,k}}{\prod_{k=1}^{n_i} r_{i,k}} \right)^{2-\dcle} \\
&\le c_2 r^{2n} (\min_{i \neq i'} |z_i - z_{i'}|)^{-2na} \left(\prod_{i=2}^n \frac{1}{\min_{i' < i} |z_i - z_{i'}|}\right)^{2-\dcle} .
\end{split}
\end{equation*}
This proves~\eqref{eqn:ball_mass_ubd} for a fixed $M$. Finally, note that changing $r_0$ to $\wt{r}_0$ in Lemma~\ref{lem:annulus_lemma} corresponds to an extra factor of $(r_0/\wt{r}_0)^n$ in~\eqref{eq:annulus_cover_bound}. This implies that we can take $E_M = E_{j_0}$ with $j_0 \asymp \log M$ as $M \to \infty$ in~\eqref{eqn:ball_mass_ubd}, concluding that $\p[(E_M)^c] = O(M^{-b})$.
\end{proof}

\begin{proof}[Proof of Proposition~\ref{prop:cle_measure_moments}]
Let $E_M$, $M>1$, be the events from Lemma~\ref{lem:measure_has_density}.  Then Lemma~\ref{lem:measure_has_density} implies that $\E[ (\meas{\cdot}{\Upsilon_{\ge\delta}} \otimes \cdots \otimes \meas{\cdot}{\Upsilon_{\ge\delta}}) \one_{E_M}]$ has a density $g$ with respect to Lebesgue measure satisfying the bound~\eqref{eqn:cle_measure_density_ubd}.  Suppose that we have $z \in D$ and $r > 0$ so that $B(z,r) \subseteq K$.  Then
\[ \E[ \meas{B(z,r)}{\Upsilon_{\ge\delta}}^n \, \one_{E_M}] = \int_{B(z,r)^n} g(z_1,\ldots,z_n) \, dz_1 \cdots dz_n.\]
Inserting the bound~\eqref{eqn:cle_measure_density_ubd} into the above expression and then integrating implies the result.
\end{proof}

\subsection{Lower bound}

We conclude the section with proving the lower bound Proposition~\ref{pr:measure_lb}. We first prove that the measure $\meas{\cdot}{\Upsilon}$ is a.s.\ positive.

\begin{lemma}\label{le:meas_nonzero}
Suppose that $\Gamma$ is a $\CLE_{\kappa}$ on a simply connected domain $D \subseteq \C$, let $\Upsilon$ be its outermost gasket, and let $\meas{\cdot}{\Upsilon}$ be the volume measure on $\Upsilon$. Then $\meas{D}{\Upsilon} > 0$ a.s.
\end{lemma}

\begin{proof}
 We have that $\meas{D}{\Upsilon}$ is not identically zero, therefore $p = \p[\meas{D}{\Upsilon} > 0] > 0$, and by the conformal covariance this probability does not depend on $D$. By repeatedly applying the Markov property of the measure we conclude that $p=1$.
\end{proof}

\begin{proof}[Proof of Proposition~\ref{pr:measure_lb}]
We use the notation introduced in Section~\ref{subsec:independence_across_scales}. In particular, for each $z \in K$ and $j > -\log_2(\dist(K,\partial D))$, let $Q_{z,j} \in \CS_j$ be the unique square containing $z$. Let $A_{z,j} = A_{Q_{z,j}} = 2Q_{z,j} \setminus (3/2)Q_{z,j}$.

For some $M>1$ large, we define for each $z,j'$ an event $F_{z,j'}$ that depends only on the CLE configuration within $A_{z,j'}$. In the case $\kappa \in (8/3,4]$ let $F_{z,j'}$ be the event that the following holds.
\begin{itemize}
 \item There are at most $M$ crossings of loops of $\Gamma$ across $A_{z,j'}$.
 \item Suppose that $\gamma_i,\gamma_{i+1}$ are two consecutive crossings, let $U_i$ be the corresponding connected component of $A_{z,j'} \setminus (\gamma_i \cup \gamma_{i+1})$, and let $V_i \subseteq U_i$ be the set of points that are connected to $\gamma_i,\gamma_{i+1}$ by an admissible path within $U_i$. Then $\meas{V_i}{\Upsilon} \ge M^{-1}2^{-j'\dcle}$ where $\Upsilon$ is the gasket that intersects the component $V_i$.
\end{itemize}
In the case $\kappa \in (4,8)$, we need a slightly stronger event (see Figure~\ref{fi:region_containing_bubble}). Here, we let $F_{z,j'}$ be the event that the following holds.
\begin{itemize}
 \item There is a collection of at most $M$ points $\{u_i\}$ in $A_{z,j'}$ that separate $A_{z,j'}^\inside$ from $A_{z,j'}^\outside$ in each gasket $\Upsilon$.
 \item For each $u_i$, the two strands intersecting at $u_i$ intersect at a point $v_i$ near $u_i$ so that if $V_i$ is the region consisting of the components bounded between the two strands from $u_i$ to $v_i$, then $\meas{V_i}{\Upsilon} \ge M^{-1}2^{-j'\dcle}$ where $\Upsilon$ is the gasket containing the points $u_i,v_i$.
\end{itemize}
Suppose that $a,b>0$ and a compact set $K \subseteq D$ are fixed. For each $j > -\log_2(\dist(K,\partial D))$, let
\[
 E_j = \bigcap_{z \in K} \bigcup_{j < j' < (1+a)j} F_{z,j'} .
\]
By Lemma~\ref{le:meas_nonzero}, the locality and scaling covariance of the volume measure, and the independence across scales, we see that if $M$ is sufficiently large (depending on $a,b$), then $\p[(E_j)^c] \le O(e^{-bj})$.

Suppose that $2^{-j} < \dist(K,\partial D)$ and we are on the event $E_j$. Let $\Upsilon$ be a gasket of $\Gamma$, let $x \in \Upsilon \cap K$ and $w \in \partial\Bpath(x,2^{-j})$. Let $B_{x,w} \subseteq \ol{\BE}(x,2^{-j})$ be a simply connected set containing all simple admissible paths within $\ol{\BE}(x,2^{-j})$ from $x$ to $w$. Let $j < j' < (1+a)j$ be such that $F_{x,j'}$ holds. Since $w \in \partial\Bpath(x,2^{-j})$, there is an admissible path $\gamma$ from $x$ to $w$ crossing $A_{x,j'}$. Then $\gamma$ needs to cross one of the $V_i$, and we have $V_i \subseteq B_{x,w}$ (see Figure~\ref{fi:region_containing_bubble} for an illustration). Therefore
\[ \meas{B_{x,w}}{\Upsilon} \ge \meas{V_i}{\Upsilon} \ge M^{-1}2^{-j'\dcle} \ge M^{-1}2^{-j(1+a)\dcle} . \]

Taking a union bound over $j$ gives the result.
\end{proof}

\section{Minkowski content construction}
\label{sec:minkowski_content}

In this section, we prove Theorem~\ref{th:minkowski}. As in Section~\ref{subsec:main_results}, we let $\kappa \in (8/3,8)$ be fixed and let $\dcle$ be as in~\eqref{eqn:cle_dim}. For the simplicity of notation, we assume throughout this section that $\Upsilon$ is the outermost gasket of $\Gamma$. For the interior gaskets, the statement of Theorem~\ref{th:minkowski} will follow by local absolute continuity, but we note that the quantitative statements proved in this section can also be shown for the interior gaskets.

\subsection{Setup for the proof}
\label{se:setup_pf}

Recall the notations introduced in Section~\ref{subsec:independence_across_scales}.

For each $j \in \N$, we will divide the collection $\CS_j$ into the four subcollections $\CS_{j,i}$, $i=1,2,3,4$, where the squares in each $\CS_{j,i}$ do not share a vertex. We let $\CF_{j,i}$ be the $\sigma$-algebra generated by the partial exploration of the strands of $\Gamma$ that intersect $\bigcap_{Q \in \CS_{j,i}} \ol{A_Q^\outside}$, extended up until they intersect $A_Q^\inside$ for some $Q \in \CS_{j,i}$. We let $\CF_{j,i,\alpha}$ be the $\sigma$-algebra generated by $\CF_{j,i}$ and the interior link patterns in each of the $Q \in \CS_{j,i}$ (note that this also determines the gasket outside the annuli).

As explained in Section~\ref{subsec:independence_across_scales}, the conditional law of $\Gamma$ in $Q$ given $\CF_{j,i}$ is that of a multichordal \clek{} where the exterior link pattern may be random in case the strands cross some other $A_{Q'}$ with $Q' \in \CS_{j,i}$ (and therefore the conditional law may change when conditioning additionally on the link pattern in~$Q'$). But if we condition on $\CF_{j,i,\alpha}$, then the multichordal \clek{} in each $Q$ are exactly independent.

Further, let $\CS^\Upsilon_j \subseteq \CS_j$ be the squares $Q$ such that the gasket $\Upsilon$ intersects $A_Q^\inside$ or contains a non-contractible loop in $A_Q$. For some $\dsep > 0$ whose value will be chosen later, let $\CS^\separated_j \subseteq \CS_j$ be the squares $Q$ such that the marked points on the $\partial A_Q^\inside$ have distance at least $\dsep 2^{-j}$ from each other. We let $\CS^{\Upsilon,\separated}_j = \CS^\Upsilon_j \cap \CS^\separated_j$ and $\CS^{\Upsilon,\separated}_{j,i} = \CS^{\Upsilon,\separated}_j \cap \CS_{j,i}$.

\begin{lemma}\label{le:good_superbox}
 For each $q \in (0,1)$, $b>0$, and compact set $K \subseteq D$ there exist $\dsep > 0$, $M > 1$, and $c>0$ such that the following holds. For each $k_0,k \in \N$ and $Q \in \CS_k$ with $Q \subseteq K$, let $G$ be the event that for at least $q$ fraction of $j = k_0,\ldots,k$, if we let $Q_{j} \in \CS_{j}$ be the square containing $Q$, then
 \[ Q_{j} \in \CS^{\separated}_{j} \quad\text{and}\quad \meas{Q_{j}}{\Upsilon} \le M2^{-j\dcle} . \]
 Then $\p[G^c] \le ce^{-b(k-k_0)}$.
\end{lemma}

\begin{proof}
 This follows from the same argument as the proof of \cite[Proposition~3.1]{my2025geouniqueness} where here we use the superpolynomial tails of the measure from Proposition~\ref{prop:cle_measure_moments}. We give a brief sketch of the proof. The condition that most scales are in $\CS^{\separated}_j$ follows from Lemma~\ref{le:separation_event}.

 To achieve the second condition, note that it suffices to show for each given sequence $k_0 \le j_1 < \cdots < j_n \le k$ that the probability that each $j_i$ fails the condition is at most $O(e^{-bn})$. Since $b$ can be taken as large as we want, this implies the lemma statement via a union bound.

 For each $k_0 \le j_0 \le k$, we decompose $Q_{j_0}$ into $R^{j_0}_{j_0} = Q_{j_0} \cap A_{Q_{j_0+1}}^\outside$ and $R^{j_0}_j = (Q_{j_0} \setminus \bigcup_{j'<j} R^{j_0}_{j'}) \cap A_{Q_{j+1}}^\outside$ for each $j>j_0$. Fix a constant $0<a<\dcle$. By Proposition~\ref{prop:cle_measure_moments}, we have
 \begin{equation}\label{eq:measure_subsequent_annuli}
  \p{\left[ \meas{R^{j_0}_j}{\Upsilon} \ge M2^{-j\dcle+a(j-j_0)} \right]} = O(M^{-1}e^{-b(j-j_0)}) .
 \end{equation}
 Moreover, this event is measurable with respect to $\CF_{Q_{j+1}}$. If we have $\meas{Q_{j_i}}{\Upsilon} \ge M^2 2^{-j_i \dcle}$ for each $i$, then there is some $j'_i \ge j_i$ with $\meas{R^{j_i}_{j'_i}}{\Upsilon} \ge M2^{-j'_i\dcle+a(j'_i-j_i)}$. The probability of this event is bounded by~\eqref{eq:measure_subsequent_annuli}, the independence across scales, and a union bound over all choices for $(j'_i)$.
\end{proof}

In the remainder of the proof, we assume that a compact set $K \subseteq D$ is fixed, and we implicitly restrict to the squares of $\CS$ that are contained in $K$.

\subsection{Comparability of the measures}
\label{se:meas_comparability}

The first step in the proof of Theorem~\ref{th:minkowski} is the following lemma.

\begin{lemma}
\label{le:comparability_on_squares}
There exist deterministic constants $0 < c_1 \leq c_2 < \infty$ such that for each $B \in \CS$ we have almost surely
\[ c_1 \limsup_{k \to \infty} \measapprox{k}{B}{\Upsilon} \le \meas{B}{\Upsilon} \le c_2 \liminf_{k \to \infty} \measapprox{k}{B}{\Upsilon} . \]
\end{lemma}

Note that since this is an almost sure statement, it does not depend on the choice of $B$ due to scaling and local absolute continuity (Lemma~\ref{le:cle_abs_cont}).

In the remainder, we write
\[
 \meas{\CS^\star_\star}{\Upsilon} = \sum_{Q \in \CS^\star_\star} \meas{Q}{\Upsilon}
\]
where $\star$ can stand for any of the symbols we introduced in Section~\ref{se:setup_pf}.

\begin{lemma}\label{le:approxmeas_moments}
 Let $B \in \CS_1$. For each $n \in \N$ there exists $c_n > 1$ such that
 \[
  c_n^{-1} \le \E[\measapprox{k}{B}{\Upsilon}^n] \le c_n
  \quad\text{for each } k > 0 .
 \]
\end{lemma}

\begin{proof}
 We begin by proving the statement for the variant~\eqref{it:box_count} which by the comparability implies the same for~\eqref{it:box_count}--\eqref{it:minkowski}.

 The statement for $n=1$ follows from \cite[Theorem~2]{ssw2009radii}. For $n>1$, by the proof of Lemma~\ref{lem:measure_has_density} there are events $E_M$ with $\p[E_M] \geq 1-O(M^{-b})$ so that for $Q_1,\ldots,Q_n \in \CS_k$ distinct we have
 \begin{multline*}
  \E{\left[ \prod_{i=1}^n \one_{\Upsilon \cap Q_i \neq \varnothing} \, \one_{E_M} \right]} \\
  \leq M 2^{-k(2-\dcle)n} \bigl( \min_{i \neq i'} |\cen(Q_i)-\cen(Q_{i'})|\bigr)^{-a} \prod_{i=2}^n \bigl( \min_{i' < i} |\cen(Q_i)-\cen(Q_{i'})| \bigr)^{\dcle - 2} .
 \end{multline*}
 where $\cen(Q)$ denotes the center of a square $Q$. This implies $\E[\measapprox{k}{B}{\Upsilon}^n \, \one_{E_M}] \lesssim M$.

 We now turn to the variants~\eqref{it:geo_count}--\eqref{it:res_count}. Let $\star = \mathrm{geo},\mathrm{res}$ and $\alpha = \geoexp$ resp.\ $\alpha = \resexp$. We define
 \[
  \measapprox{\mathrm{p},\delta}{B}{\Upsilon} = \delta^{\dcle} \max\left\{ n \in \N : \begin{array}{cc} \text{There exist $x_1,\ldots,x_n \in \Upsilon$ such that} \\ \text{$B_\star(x_i,\delta^\alpha) \subseteq B \cap \Upsilon$ and are pairwise disjoint}\end{array} \right\} .
 \]
 Then for each $j \in \N$ we have
 \[
  \measapprox{\mathrm{p},\delta}{\CS_j}{\Upsilon} \le \measapprox{\mathrm{p},\delta}{B}{\Upsilon} \le \measapprox{\delta}{B}{\Upsilon} \le \measapprox{\delta}{\CS_j}{\Upsilon} .
 \]
 It is shown in \cite[Section~5.4]{amy2025tightness} that for some $k_0 > 0$ we have $0 < \E[ \measapprox{\mathrm{p},k_0}{B}{\Upsilon}^n ] \le \E[ \measapprox{k_0}{B}{\Upsilon}^n ] < \infty$ for each $n \in \N$. The same is true in the case $\kappa \in (8/3,4)$ and $\star = \mathrm{geo}$, see \cite[Definition~1.2(IV)]{mt-geodesic} and turning it into a quantitative local estimate using the same resampling argument.

 For each $j \in \N$ and $Q \in \CS_j$ we have
 \[\begin{split}
  \E[\measapprox{j+k_0}{Q}{\Upsilon}]
  &= \p[\Upsilon \cap 2Q \neq \varnothing] \E[\measapprox{j+k_0}{Q}{\Upsilon} \mid \Upsilon \cap 2Q \neq \varnothing] \\
  &\asymp \p[\Upsilon \cap 2Q \neq \varnothing] 2^{-j\dcle}\E[\measapprox{k_0}{B}{\Upsilon}] \\
  &\asymp 2^{-2j}\E[\measapprox{k_0}{B}{\Upsilon}]
 \end{split}\]
 where the last identity is \cite[Theorem~2]{ssw2009radii}, and the second identity follows because conditionally on a \clek{} exploration up until hitting $2Q$, the remainder is a monochordal \clek{} which can be compared to a \clek{} by the argument in Lemma~\ref{lem:gasket_measure_first_moment_mcle}. The same holds for $\measapprox{\mathrm{p},k}{B}{\Upsilon}$. This implies the lemma statement for $n=1$.

 For $n>1$, the proof of Lemma~\ref{lem:measure_has_density} shows that for $a>0$ small we have for $Q_1,\ldots,Q_n \in \CS_j$ with $\min_{i \neq i'} |\cen(Q_i)-\cen(Q_{i'})| \ge 2^{-(1-a)j}$ we have
 \begin{multline*}
  \E{\left[ \prod_{i=1}^n \measapprox{j+k_0}{Q_i}{\Upsilon} \, \one_{E_M} \right]} \\
  \leq M 2^{-2jn} \bigl( \min_{i \neq i'} |\cen(Q_i)-\cen(Q_{i'})|\bigr)^{-a} \prod_{i=2}^n \bigl( \min_{i' < i} |\cen(Q_i)-\cen(Q_{i'})| \bigr)^{\dcle - 2} .
 \end{multline*}
 Combining this with the above, we conclude that $\E[\measapprox{j+k_0}{B}{\Upsilon}^n \, \one_{E_M}] \lesssim M$.
\end{proof}

\begin{lemma}\label{le:approx_additivity}
 For each $a>0$, $b>0$ there exists $c>0$ such that the following holds for each $k \ge j$.
 \[
  \p{\left[ \measapprox{k}{\CS_j}{\Upsilon}-\measapprox{k}{B}{\Upsilon} > 2^{j-k(\dcle-1-a)} \right]} \le ce^{-bk} .
 \]
\end{lemma}

\begin{proof}
 The measures~\eqref{it:box_count}--\eqref{it:gps_box_count} are exactly additive. The discrepancy in~\eqref{it:minkowski} is bounded by $2^{-k(\dcle-2)}$ times the area of the $2^{-k}$-neighborhood of $\bigcup_{Q \in \CS^\Upsilon_j} \partial Q$ which is at most $O(2^{-(k-j)(\dcle-1)}\measapprox{j}{B}{\Upsilon})$.

 We turn to~\eqref{it:geo_count}--\eqref{it:res_count}. Let $\star = \mathrm{geo},\mathrm{res}$ and $\alpha = \geoexp$ resp.\ $\alpha = \resexp$. Let $a>0$ be small, and let $G^1_k$ be the event that $\abs{x-y} \le 2^{-k(1-a)}$ for each $x,y \in \Upsilon \cap B$ with $\met{x}{y}{\Gamma} < 2^{-k\alpha}$. Let $G^2_k$ be the event that each $Q \in \CS_k$ can be covered by at most $2^{ak}$ many $B_\star$-balls of radius $2^{-k\alpha}$. It is shown in \cite{my2025geouniqueness,my2025resuniqueness} that $\p[(G^1_k)^c] = o^\infty(e^{-k})$ and $\p[(G^2_k)^c] = o^\infty(e^{-k})$. The same is true in the case $\kappa \in (8/3,4)$ and $\star = \mathrm{geo}$ by the same resampling and independence across scales argument.

 Suppose that we are on the event $G^1_k \cap G^2_k$. In $\measapprox{k}{\CS_j}{\Upsilon}$ we are only double counting the balls that intersect $\partial Q$ for some $Q \in \CS^\Upsilon_j$. The number of squares in the Euclidean $2^{-k(1-a)}$-neighborhoods of all $\partial Q$ is at most $2^{-j+k(1+a)}\#\CS^\Upsilon_j$, and therefore they are covered by at most $2^{-j+k(1+2a)}\#\CS^\Upsilon_j$ many $B_\star$-balls of radius $2^{-k\alpha}$.
 To conclude, we simply bound the number $\#\CS^\Upsilon_j$ by $2^{2j}$. (This is not optimal, but suffices for our purpose.)
\end{proof}

In the following, we let $N_j$ be the number of squares in $\CS^{\Upsilon,\separated}_j$ and for $i=1,2,3,4$ let $N_{j,i}$ be the number of squares in $\CS^{\Upsilon,\separated}_{j,i}$.

\begin{lemma}\label{le:meas_large_dev}
 For each $\dsep > 0$ and $b>0$ there exist $M>1$ and $c>0$ such that the following hold for each~$k \ge j$.
 \[
  \p{\left[ M^{-1}2^{-j\dcle}N_{j,i} \le \meas{\CS^{\Upsilon,\separated}_{j,i}}{\Upsilon} \le M2^{-j\dcle}N_{j,i} \mid \CF_{j,i,\alpha} \right]} \ge 1-cN_{j,i}^{-b}
 \]
 and
 \[
  \p{\left[ M^{-1}2^{-j\dcle}N_{j,i} \le \measapprox{k}{\CS^{\Upsilon,\separated}_{j,i}}{\Upsilon} \le M2^{-j\dcle}N_{j,i} \mid \CF_{j,i,\alpha} \right]} \ge 1-cN_{j,i}^{-b} .
 \]
\end{lemma}

Lemma~\ref{le:meas_large_dev} follows from the following two lemmas which give respectively the lower and upper bound.

\begin{lemma}\label{le:meas_cond_lb}
 For each $\dsep > 0$ there exist $M>1$ and $p>0$ such that the following hold for each~$Q \in \CS^{\Upsilon,\separated}_{j,i}$ and~$k \ge j$.
 \[
  \p{\left[ \meas{Q}{\Upsilon} \ge M^{-1}2^{-j\dcle} \mid \CF_{j,i,\alpha} \right]} \ge p
 \]
 and
 \[
  \p{\left[ \measapprox{k}{Q}{\Upsilon} \ge M^{-1}2^{-j\dcle} \mid \CF_{j,i,\alpha} \right]} \ge p .
 \]
\end{lemma}

\begin{proof}
By \cite[Lemma~3.10 and Theorem~1.7]{amy-cle-resampling}, there is some $p>0$ (depending only on $\dsep$) such that for each marked point configuration that is $\dsep$-separated, each interior link pattern, and each choice of boundary interval, the probability that the gasket connected to that boundary interval intersects $Q$ is at least $p$. The measure $\meas{Q}{\Upsilon}$ is almost surely positive on this event by Lemma~\ref{le:meas_nonzero}, and by \cite[Proposition~6.3]{amy-cle-resampling} its law depends continuously on the marked domain. This together with the scale covariance gives the desired statement for $\meas{Q}{\Upsilon}$. The statement for $\measapprox{k}{Q}{\Upsilon}$ follows similarly from Lemma~\ref{le:approxmeas_moments}.
\end{proof}

\begin{lemma}\label{le:meas_cond_ub}
 For each $\dsep > 0$ and $b>0$ there exists $c>0$ such that the following hold for each $Q \in \CS^{\Upsilon,\separated}_{j,i}$ and $k \ge j$.
 \[
  \p[ \meas{Q}{\Upsilon} \ge M2^{-j\dcle} \mid \CF_{j,i,\alpha} ] \le cM^{-b}
 \]
 and
 \[
  \p[ \measapprox{k}{Q}{\Upsilon} \ge M2^{-j\dcle} \mid \CF_{j,i,\alpha} ] \le cM^{-b} .
 \]
\end{lemma}

\begin{proof}
The upper tail for $\meas{Q}{\Upsilon}$ follows from Proposition~\ref{prop:cle_measure_moments} by the same argument as in the proof of Lemma~\ref{lem:gasket_measure_first_moment_mcle}. The upper tail for $\measapprox{k}{Q}{\Upsilon}$ follows similarly from Lemma~\ref{le:approxmeas_moments}.
\end{proof}

\begin{lemma}\label{le:comparability_tails}
 There exists $a>0$ such that the following is true. Let $B \in \CS$. For each $b>0$ there exists $M>1$ such that
 \[
  \p{\left[ \frac{\measapprox{k}{B}{\Upsilon}}{\meas{B}{\Upsilon}} \notin [M^{-1},M] ,\, \Upsilon \cap (1-2^{-ak})B \neq \varnothing \right]} = O(e^{-bk}) .
 \]
\end{lemma}

\begin{proof}
Let $G_k$ be the event that for each $Q \in \CS_{k/2}$ with $Q \subseteq B$, for at least $3/4$ fraction of $j=k/3,\ldots,k/2$ we have $Q_j \in \CS^{\separated}_j$ where $Q_j$ is the square containing $Q$. Then (for a suitable choice of $\dsep$) we have $\p[(G_k)^c] = O(e^{-bk})$ by Lemma~\ref{le:good_superbox}. Suppose that we are on the event $G_k$. Then
\[
 \sum_{j=k/3,\ldots,k/2} \meas{\CS^{\Upsilon,\separated}_{j}}{\Upsilon} = \sum_{j=k/3,\ldots,k/2} \sum_{Q \in \CS^\Upsilon_{k/2}} \meas{Q}{\Upsilon} \one_{Q_{j} \in \CS^{\Upsilon,\separated}_{j}} \ge \frac{3}{4}\frac{k}{6} \meas{B}{\Upsilon} .
\]
The same is true for $\measapprox{k}{\cdot}{\Upsilon}$. In particular, for at least $1/2$ fraction of $k/3 \le j \le k/2$ we have $\meas{\CS^{\Upsilon,\separated}_{j}}{\Upsilon} \ge \frac{1}{2} \meas{B}{\Upsilon}$ and similarly $\measapprox{k}{\CS^{\Upsilon,\separated}_{j}}{\Upsilon} \ge \frac{1}{2} \measapprox{k}{B}{\Upsilon}$.

By Lemma~\ref{le:meas_large_dev} we have for large enough $M$ that
\[
 \p{\left[ \frac{\measapprox{k}{\CS^{\Upsilon,\separated}_j}{\Upsilon}}{\meas{\CS^{\Upsilon,\separated}_j}{\Upsilon}} \notin [M^{-1},M] \right]} = O(\E[N_j^{-b}]) .
\]
To lower bound $N_j$, note that
\[ \p[ \meas{B}{\Upsilon} \le 2^{-4ak} ,\, \Upsilon \cap (1-2^{-ak})B \neq \varnothing ] = o^\infty(e^{-k}) \]
by Proposition~\ref{pr:measure_lb}. Moreover, by Lemma~\ref{le:meas_cond_ub} we have
\[
 \p{\left[ N_j < 2^{j\dcle-5ak} ,\, \meas{\CS^{\Upsilon,\separated}_{j}}{\Upsilon} \ge \frac{1}{2} \meas{B}{\Upsilon} \ge 2^{-4ak} \right]} = o^\infty(e^{-k}) .
\]
Supposing that $a < 1/15$ (say), we conclude that
\[
 \p{\left[ \frac{\measapprox{k}{\CS^{\Upsilon,\separated}_j}{\Upsilon}}{\meas{\CS^{\Upsilon,\separated}_j}{\Upsilon}} \notin [M^{-1},M] ,\, \meas{\CS^{\Upsilon,\separated}_{j}}{\Upsilon} \ge \frac{1}{2} \meas{B}{\Upsilon} \right]} = O(e^{-bk})
\]
for each $j=k/3,\ldots,k/2$. Incorporating Lemma~\ref{le:approx_additivity} and taking a union bound over $j$ gives the result.
\end{proof}

\begin{proof}[Proof of Lemma~\ref{le:comparability_on_squares}]
 As discussed above, it suffices to prove this for a specific choice of $B \in \CS$. This follows immediately from Lemma~\ref{le:comparability_tails}.
\end{proof}

\subsection{Identification of the limit}
\label{se:meas_limit}

Next, we are going to show that the statement of Theorem~\ref{th:minkowski} holds with some additional constants $(c_k)$ and for the convergence in probability instead of almost surely.

\begin{lemma}\label{le:in_prob}
There exists a deterministic sequence $(c_k)$ with $0 < \liminf_{k\to\infty} c_k \le \limsup_{k\to\infty} c_k < \infty$ such that for each $B \in \CS$ we have $c_k\measapprox{k}{B}{\Upsilon} \to \meas{B}{\Upsilon}$ in probability.
\end{lemma}

\begin{proof}
We show that each subsequence $(k_\ell)$ contains a further subsequence $(k_{\ell_m})$ for which there is a constant $c>0$ such that $\measapprox{k_{\ell_m}}{B}{\Upsilon} \to c\meas{B}{\Upsilon}$ in probability for each $B \in \CS$.

By Lemma~\ref{le:comparability_on_squares}, there is a subsequence $(k_{\ell_m})$ along which the law of $(\Gamma, (\meas{B}{\Upsilon}, \measapprox{k_{\ell_m}}{B}{\Upsilon})_{B \in \CS})$ converges weakly. Let $(\Gamma, (\meas{B}{\Upsilon}, \wt{\mu}(B))_{B \in \CS})$ denote the weak limit. Note that $\wt{\mu}$ is exactly additive by Lemma~\ref{le:approx_additivity}. We aim to show that there is a constant $c>0$ such that $\wt{\mu}(B) = c\meas{B}{\Upsilon}$ for each $B \in \CS$.

Let
\[
 c_* = \sup \left\{ c>0 : \p[ \wt{\mu}(B) < c\meas{B}{\Upsilon} ] = 0 \text{ for each $B \in \CS$} \right\} ,
\]
and
\[
 c^* = \inf \left\{ c>0 : \p[ \wt{\mu}(B) > c\meas{B}{\Upsilon} ] = 0 \text{ for each $B \in \CS$} \right\} .
\]
By Lemma~\ref{le:comparability_on_squares} we have $0 < c_* \le c^* < \infty$. Suppose that $c_* < c^*$. We will derive a contradiction by showing that for some constant $\delta>0$ we have either
\[
 \wt{\mu}(B) \ge (c_*+\delta)\meas{B}{\Upsilon} \text{ for each } B \in \CS
 \quad\text{or}\quad
 \wt{\mu}(B) \le (c^*-\delta)\meas{B}{\Upsilon} \text{ for each } B \in \CS .
\]
Fix some $B_0 \in \CS$. We can assume that for each $j \in \N$ also the law of $\frac{\measapprox{k_\ell}{2^{-j}B_0}{2^{-j}\Upsilon}}{\meas{2^{-j}B_0}{2^{-j}\Upsilon}} = \frac{\measapprox{k_\ell-j}{B_0}{\Upsilon}}{\meas{B_0}{\Upsilon}}$ converges weakly (switching to a subsequence if necessary). Let $\CJ \subseteq \N$ be the set of scales $j$ such that
\[
 \limsup_{\ell\to\infty} \p{\left[ \frac{\measapprox{k_\ell}{2^{-j}B_0}{2^{-j}\Upsilon}}{\meas{2^{-j}B_0}{2^{-j}\Upsilon}} \le \frac{c_*+c^*}{2} \mmiddle| \Upsilon \cap B_0 \neq \varnothing \right]} \ge \frac{1}{2} .
\]
By Lemmas~\ref{le:cle_abs_cont}--\ref{le:mcle_abs_cont}, for each $\dsep > 0$ there is $p > 0$ such that for each $j \in \CJ$ we have
\begin{equation}\label{eq:ub_improve}
 \liminf_{\ell\to\infty} \min_{Q \in \CS^{\Upsilon,\separated}_{j,i}} \p{\left[ \frac{\measapprox{k_\ell}{Q}{\Upsilon}}{\meas{Q}{\Upsilon}} \le \frac{c_*+c^*}{2} \mmiddle| \CF_{j,i,\alpha} \right]} \ge p ,
\end{equation}
and for each $j \in \N \setminus \CJ$ we have
\begin{equation}\label{eq:lb_improve}
 \liminf_{\ell\to\infty} \min_{Q \in \CS^{\Upsilon,\separated}_{j,i}} \p{\left[ \frac{\measapprox{k_\ell}{Q}{\Upsilon}}{\meas{Q}{\Upsilon}} \ge \frac{c_*+c^*}{2} \mmiddle| \CF_{j,i,\alpha} \right]} \ge p .
\end{equation}

Let $j_0 \in \N$ and let $G^1_{j_0}$ be the event that for each $Q \in \CS_{2j_0}$, for at least $7/8$ fraction of $j=j_0,\ldots,2j_0$, if we let $Q_j \in \CS_j$ be the square containing $Q$, then
\[ Q_{j} \in \CS^{\separated}_{j} \quad\text{and}\quad \meas{Q_{j}}{\Upsilon} \le M2^{-j\dcle} . \]
By Lemma~\ref{le:good_superbox}, we can choose $\dsep > 0$, $M>1$ so that $\p[(G^1_{j_0})^c] = O(e^{-bj_0})$.

Let $B \in \CS$ and let $\CS^g_j \subseteq \CS^{\Upsilon,\separated}_j$ be the subset of squares $Q \subseteq B$ such that $\meas{Q}{\Upsilon} \le M2^{-j\dcle}$. On the event $G^1_{j_0}$ we have
\[
 \sum_{j=j_0,\ldots,2j_0} \meas{\CS^g_j}{\Upsilon} \ge \frac{7}{8}j_0 \meas{B}{\Upsilon} .
\]
In particular, for at least $3/4$ fraction of $j_0 \le j \le 2j_0$ we have $\meas{\CS^g_{j}}{\Upsilon} \ge \frac{1}{2} \meas{B}{\Upsilon}$.

Suppose now that $\#(\CJ \cap [j_0,2j_0]) \ge j_0/2$. Then on $G^1_{j_0}$ at least $1/4$ fraction of $j_0 \le j \le 2j_0$ is in~$\CJ$ and satisfies $\meas{\CS^g_{j}}{\Upsilon} \ge \frac{1}{2} \meas{B}{\Upsilon}$. For $j \in \CJ$, let $E_j$ denote the event that $\meas{\CS^g_{j}}{\Upsilon} \ge \frac{1}{2} \meas{B}{\Upsilon}$.

Fix some $0 < a < \dcle$ and note that $\p[ \meas{B}{\Upsilon} \le 2^{-aj} ,\, \Upsilon \cap \lambda B \neq \varnothing ] = o^\infty(e^{-j})$ for each $\lambda \in (0,1)$ by Proposition~\ref{pr:measure_lb}. On the event $\meas{\CS^g_{j}}{\Upsilon} \ge \frac{1}{2} \meas{B}{\Upsilon} \ge 2^{-aj}$ we necessarily have $N_j \ge M^{-1}2^{j(\dcle-a)}$ (where $N_j$ is defined above Lemma~\ref{le:meas_large_dev}).

By~\eqref{eq:ub_improve} and Lemma~\ref{le:meas_cond_lb}, there is $q_2>0$ and $M>1$ (depending only on $\dsep$) such that the following hold. Let $G^2_{j,i}$ be the event that for at least $q_2$ fraction of $Q \in \CS^{\Upsilon,\separated}_{j,i}$ we have
\begin{equation}\label{eq:improved_squares}
 \frac{\measapprox{k_\ell}{Q}{\Upsilon}}{\meas{Q}{\Upsilon}} \le \frac{c_*+c^*}{2}
 \quad\text{and}\quad
 \meas{Q}{\Upsilon} \ge M^{-1}2^{-j\dcle} .
\end{equation}
Then $\p[ (G^2_{j,i})^c \mid \CF_{j,i,\alpha} ] = o^\infty(N_{j,i}^{-1})$ for sufficiently large $\ell$. We conclude that, if we let $G^2_j$ be the event that~\eqref{eq:improved_squares} is satisfied for at least $q_2$ fraction of $Q \in \CS^{\Upsilon,\separated}_{j}$, then
\[
 \p[ E_j \setminus G^2_j ] = o^\infty(e^{-j})
 \quad\text{for sufficiently large $\ell$.}
\]
Let $\CS^2_j \subseteq \CS^{\Upsilon,\separated}_{j}$ be the squares where~\eqref{eq:improved_squares} hold. On $E_j \cap G^2_j$ we have that
\[
 \meas{\CS^2_j}{\Upsilon} \ge q_2 N_j M^{-1}2^{-j\dcle} \ge q_2 M^{-2}\meas{\CS^g_j}{\Upsilon} \ge \frac{q_2}{2} M^{-2} \meas{B}{\Upsilon} .
\]
Finally, by the definition of $c^*$ we have for each $\varepsilon > 0$ and each $Q \in \CS_j$ that
\[
 \lim_{\ell\to\infty} \p{\left[ \frac{\measapprox{k_\ell}{Q}{\Upsilon}}{\meas{Q}{\Upsilon}} \ge c^*+\varepsilon \right]} = 0 .
\]
Therefore we have
\[\begin{split}
 \measapprox{k_\ell}{B}{\Upsilon}
 &= \measapprox{k_\ell}{\CS^2_j}{\Upsilon} + \measapprox{k_\ell}{\CS_j \setminus \CS^2_j}{\Upsilon} \\
 &\le \frac{c_*+c^*}{2} \meas{\CS^2_j}{\Upsilon} + (c^*+\varepsilon) \meas{\CS_j \setminus \CS^2_j}{\Upsilon} \\
 &\le (c^*+\varepsilon) \meas{B}{\Upsilon} - \frac{c^*-c_*}{2}\frac{q_2}{2} M^{-2} \meas{B}{\Upsilon} \\
 &\eqdef (c^*-\delta) \meas{B}{\Upsilon}
\end{split}\]
on an event with probability tending to $1$. This holds for each $B \in \CS$, hence we have shown that $\wt{\mu}(B) \le (c^*-\delta)\meas{B}{\Upsilon}$ in case the set of $j_0$ with $\#(\CJ \cap [j_0,2j_0]) \ge j_0/2$ is infinite.

In case the set of $j_0$ with $\#(\CJ \cap [j_0,2j_0]) \ge j_0/2$ is finite, a symmetric argument using~\eqref{eq:lb_improve} shows that $\wt{\mu}(B) \ge (c_*+\delta)\meas{B}{\Upsilon}$ for each $B \in \CS$. This concludes the proof.
\end{proof}

\subsection{Improving the convergence}
\label{se:convergence_improve}

Next, we improve Lemma~\ref{le:in_prob} to almost sure convergence.

\begin{lemma}\label{le:nonconstant_as}
 There exists a deterministic sequence $(c_k)$ with $0 < \liminf_{k\to\infty} c_k \le \limsup_{k\to\infty} c_k < \infty$ such that for each $B \in \CS$ and $n \in \N$ we have $c_k\measapprox{k}{B}{\Upsilon} \to \meas{B}{\Upsilon}$ almost surely along $k \in 2^{-n}\N$. Moreover, for each $\varepsilon > 0$ we have
 \[
  \p{\left[ \abs{c_k\measapprox{k}{B}{\Upsilon}-\meas{B}{\Upsilon}} \ge \varepsilon \right]} = o^\infty(e^{-k}) .
 \]
\end{lemma}

\begin{proof}
Let $B \in \CS$. Throughout the proof $b$ denotes a large constant whose value may change from line to line.

Let $q \in (0,1)$. By Lemma~\ref{le:good_superbox}, we can choose $\dsep > 0$, $M>1$ so that the following hold. Let $G^1_k$ be the event that for each $Q \in \CS_{2k}$ with $Q \subseteq B$, for at least $q$ fraction of $j = k,\ldots,2k$, if we let $Q_{j} \in \CS_{j}$ be the square containing $Q$, then
\[ Q_{j} \in \CS^{\separated}_{j} \quad\text{and}\quad \meas{Q_{j}}{\Upsilon} \le M2^{-j\dcle} . \]
Then $\p[(G^1_k)^c] = O(e^{-bk})$.

Let $\CS^g_j \subseteq \CS^{\Upsilon,\separated}_j$ be the subset of squares $Q \subseteq B$ such that $\meas{Q}{\Upsilon} \le M2^{-j\dcle}$. On the event~$G^1_k$ we have
\[
 \sum_{j=k,\ldots,2k} \meas{\CS^g_j}{\Upsilon} \ge qk\meas{B}{\Upsilon} .
\]
In particular, for at least $7/8$ fraction of $k \le j \le 2k$ we have $\meas{\CS^g_{j}}{\Upsilon} \ge q_1\meas{B}{\Upsilon}$ where $1-q_1 = 8(1-q)$. Similarly, for at least $7/8$ fraction of $k \le j \le 2k$ we have $\measapprox{3k}{\CS^g_{j}}{\Upsilon} \ge q_1\measapprox{3k}{B}{\Upsilon}$.

For each $j=k,\ldots,2k$ we consider the following events. Let $a>0$ be a small constant and let $G^2_j$ be the event that $M_2^{-1} \le \frac{\measapprox{3k}{Q}{\Upsilon}}{\meas{Q}{\Upsilon}} \le M_2$ for each $Q \in \CS_j$ with $\Upsilon \cap (1-2^{-aj})Q \neq \varnothing$. By Lemma~\ref{le:comparability_tails}, we can fix $M_2>1$ (depending only on $b$) such that $\p[(G^2_j)^c] = O(e^{-bk})$.

On the other hand, if we let $G^3_j$ be the event that $\measapprox{3k}{Q \setminus (1-2^{-aj})Q}{\Upsilon} \le 2^{-j(1+a/2)\dcle}$ for each $Q \in \CS_j$, then $\p[(G^3_j)^c] = o^\infty(e^{-j})$ by Lemma~\ref{le:approxmeas_moments}.

Next, by Lemma~\ref{le:in_prob} we have
\[
 \frac{c_{3k-j}\measapprox{3k}{2^{-j}B}{2^{-j}\Upsilon}}{\meas{2^{-j}B}{2^{-j}\Upsilon}} = \frac{c_{3k-j}\measapprox{3k-j}{B}{\Upsilon}}{\meas{B}{\Upsilon}} \to 1 \quad\text{in probability.}
\]
In particular, for each $\varepsilon > 0$,
\begin{equation}\label{eq:in_prob_scaled}
 \max_{j=k,\ldots,2k} \p{\left[ \frac{c_{3k-j}\measapprox{3k}{2^{-j}B}{2^{-j}\Upsilon}}{\meas{2^{-j}B}{2^{-j}\Upsilon}} \notin [1-\varepsilon,1+\varepsilon] \right]} \to 0 .
\end{equation}
Let
\[
 G^4_j = \left\{ 1-\varepsilon \le \frac{c_{3k-j}\measapprox{3k}{Q}{\Upsilon}}{\meas{Q}{\Upsilon}} \le 1+\varepsilon
 \quad\text{for at least $1-\varepsilon$ fraction of $Q \in \CS^{\Upsilon,\separated}_{j}$} \right\} .
\]
By~\eqref{eq:in_prob_scaled} and Lemmas~\ref{le:cle_abs_cont}--\ref{le:mcle_abs_cont}, for any choice of $\varepsilon > 0$ we have for $k$ sufficiently large $\p[(G^4_j)^c] = O(e^{-bk})$ (at least when restricted to the event that $N_j \ge 2^{aj}$ for some $a>0$ as in the proofs of Lemma~\ref{le:comparability_tails} and~\ref{le:in_prob}).

Finally, let $G^5_j$ be the event that at least $1/2$ fraction of $Q \in \CS^{\Upsilon,\separated}_{j}$ is in $\CS^g_j$, and at least $M^{-1}$ fraction of $Q \in \CS^{\Upsilon,\separated}_{j}$ is in $\CS^g_j$ and satisfies $\meas{Q}{\Upsilon} \ge M^{-1}2^{-j\dcle}$. Then $\p[(G^5_j)^c] = O(e^{-bk})$ by Lemma~\ref{le:meas_cond_lb}.

Now suppose that we are on the event $G^2_j \cap G^3_j \cap G^4_j \cap G^5_j$. Let $\CS^4_j \subseteq \CS^g_j$ be the squares occurring in the event $G^4_j$. In particular, at most $2\varepsilon$ fraction of the squares in $\CS^g_j$ is not in $\CS^4_j$, and $\meas{\CS^g_j \setminus \CS^4_j}{\Upsilon} \le 2\varepsilon M^3 \meas{\CS^g_j}{\Upsilon}$. Moreover, each $\CS^g_j \setminus \CS^4_j$ satisfies either $\measapprox{3k}{Q}{\Upsilon} \le M_2 \meas{Q}{\Upsilon}$ (if $\Upsilon \cap (1-2^{-aj})Q \neq \varnothing$) or $\measapprox{3k}{Q}{\Upsilon} \le 2^{-j(1+a/2)\dcle}$ (if $\Upsilon \cap (1-2^{-aj})Q = \varnothing$). Therefore
\[\begin{split}
 \measapprox{3k}{\CS^g_j}{\Upsilon}
 &= \measapprox{3k}{\CS^4_j}{\Upsilon} + \measapprox{3k}{\CS^g_j \setminus \CS^4_j}{\Upsilon} \\
 &\le c_{3k-j}^{-1}(1+\varepsilon)\meas{\CS^4_j}{\Upsilon} + M_2 \meas{\CS^g_j \setminus \CS^4_j}{\Upsilon} +O(2^{-ja/2})\meas{\CS^g_j}{\Upsilon} \\
 &\le (c_{3k-j}^{-1}+\varepsilon M^4)\meas{\CS^g_j}{\Upsilon} .
\end{split}\]
Hence, on the event that $\measapprox{3k}{\CS^g_{j}}{\Upsilon} \ge q_1\measapprox{3k}{B}{\Upsilon}$ we have
\[
 \measapprox{3k}{B}{\Upsilon} \le (c_{3k-j}^{-1}+\varepsilon M^4)q_1^{-1}\meas{B}{\Upsilon} .
\]
Similarly, on the event that $\meas{\CS^g_{j}}{\Upsilon} \ge q_1\meas{B}{\Upsilon}$ we have
\[
 \measapprox{3k}{B}{\Upsilon} \ge (c_{3k-j}^{-1}-2\varepsilon M^3)q_1 \meas{B}{\Upsilon}
\]
where here we also used Lemma~\ref{le:approx_additivity}. Since we can take $q_1$ close to $1$ and then take $\varepsilon$ as small as we want, we conclude. (This shows also that $c_{3k} = c_{3k-j}+o(1)$ for at least $3/4$ fraction of $k \le j \le 2k$.)
\end{proof}

To complete the proof of Theorem~\ref{th:minkowski}, we need to argue that the convergence in Lemma~\ref{le:nonconstant_as} also holds when $c_k$ is taken to be a constant.

\begin{lemma}\label{le:constant_as}
 The sequence $(c_k)$ in Lemma~\ref{le:nonconstant_as} is convergent.
\end{lemma}

\begin{proof}
 First, we argue that
 \[
  c_k\E[\measapprox{k}{B}{\Upsilon}] = \E[\meas{B}{\Upsilon}]+o(1)
  \quad\text{as } k \to \infty .
 \]
 By Lemma~\ref{le:meas_cond_ub} we have
 \[
  \sup_k \E[\measapprox{k}{B}{\Upsilon}^2] < \infty .
 \]
 Hence, by Lemma~\ref{le:nonconstant_as}, we see that for each $\varepsilon > 0$,
 \[
  \E{\left[ \measapprox{k}{B}{\Upsilon} \one_{\abs{c_k\measapprox{k}{B}{\Upsilon}-\meas{B}{\Upsilon}} \ge \varepsilon} \right]}
  \le \E{\left[\measapprox{k}{B}{\Upsilon}^2\right]}^{1/2} \p{\left[ \abs{c_k\measapprox{k}{B}{\Upsilon}-\meas{B}{\Upsilon}} \ge \varepsilon \right]}^{1/2}
  = o^\infty(e^{-k}) .
 \]
 Now let $j \in \N$. Then
 \[
  \E[ \measapprox{j+k}{\CS_j}{\Upsilon} ]
  = \sum_{Q \in \CS_j} \p[ \Upsilon \cap 2Q \neq \varnothing ] \E[ \measapprox{j+k}{Q}{\Upsilon} \mid \Upsilon \cap 2Q \neq \varnothing ] .
 \]
 We have $\p[ \Upsilon \cap 2Q \neq \varnothing ] \asymp 2^{-j(2-\dcle)}$ by \cite[Theorem~2]{ssw2009radii}. Further, by conditioning on a \clek{} exploration until hitting $2Q$ and applying the same argument as above, we see that
 \[
  c_k 2^{j\dcle}\E[ \measapprox{j+k}{Q}{\Upsilon} \mid \Upsilon \cap 2Q \neq \varnothing ] = 2^{j\dcle}\E[ \meas{Q}{\Upsilon} \mid \Upsilon \cap 2Q \neq \varnothing ] + o(1) .
 \]
 We conclude that
 \[
  \E[\meas{B}{\Upsilon}] + o(1) = c_{j+k}\E[\measapprox{j+k}{B}{\Upsilon}] \le c_{j+k}\E[\measapprox{j+k}{\CS_j}{\Upsilon}] = \frac{c_{j+k}}{c_k}\E[\meas{B}{\Upsilon}] + o(1) ,
 \]
 and hence $c_{j+k} \ge c_k+o(1)$ for all $j,k$ sufficiently large.
\end{proof}

\begin{proof}[Proof of Theorem~\ref{th:minkowski}]
 For the exterior gasket $\Upsilon$, we have proved in Lemmas~\ref{le:nonconstant_as}--\ref{le:constant_as} the almost sure convergence in Theorem~\ref{th:minkowski} along sequences $k \in 2^{-n}\N$ for each $n \in \N$. The unrestricted limit follows by the monotonicity of the ball counts as we vary $k > 0$.

 The result for the interior gaskets follows by comparing the law of $\Upsilon_\CL$ to $\Upsilon$ locally using resampling (as in \cite[Section~5.2]{amy-cle-resampling}) and Lemma~\ref{le:cle_abs_cont}.
\end{proof}

\bibliographystyle{alpha}
\providecommand{\noopsort}[1]{}

\end{document}